\documentclass[12pt]{article}
\usepackage{geometry}                
\usepackage{graphicx}
\usepackage{amssymb}
\usepackage{amsmath}
\usepackage{mathtools}
\usepackage{amsthm}
\usepackage{comment}
\usepackage{tabu}
\usepackage{hyperref}
\newcommand*{\lhs}{\hspace{1.5cm}&\hspace{-1.5cm}}

\newcommand*{\NN}{\mathbb{N}}
\newcommand*{\ZZ}{\mathbb{Z}}

\DeclareMathOperator*{\PP}{\mathbb{P}}
\DeclareMathOperator*{\EE}{\mathbb{E}}

\newcommand*{\calE}{\mathcal{E}}

\newcommand*{\calO}{\mathcal{O}}

\newcommand*{\calW}{\mathcal{W}}

\newcommand*{\A}{\mathtt{A}}
\newcommand*{\B}{\mathtt{B}}

\newcommand*{\amap}{\mathtt{a}}
\newcommand*{\bmap}{\mathtt{b}}

\newcommand*{\1}[1]{\mathbf{1}_{\{#1\}}}

\newcommand*{\join}{\mathsf{J}}
\newcommand*{\tup}{}

\newcommand*{\shent}{\mathrm{H}} 
\newcommand*{\h}{\mathrm{h}} 
\newcommand*{\f}{f} 

\newcommand*{\SBM}{\mathtt{SBM}} 

\let\mb\mathbf

\newcommand*{\st}{\,:\,}

\newtheorem{lemma}{Lemma}[section]
\newtheorem{cor}[lemma]{Corollary}
\newtheorem{prop}[lemma]{Proposition}

\newtheorem{mainthm}{Theorem}
\newtheorem{mainprop}{Proposition}
\newtheorem{claim}{Claim}

\theoremstyle{definition}

\DeclareMathOperator{\Prob}{Prob}

\DeclareMathOperator{\Unif}{Unif}
\newcommand*{\limsupinf}{\limsup} 
\DeclareMathOperator{\Hom}{Hom}
\DeclareMathOperator{\Sym}{Sym}

\DeclareMathOperator{\TV}{TV}

\DeclarePairedDelimiter{\abs}{\lvert}{\rvert}
\DeclarePairedDelimiter{\floor}{\lfloor}{\rfloor}
\DeclarePairedDelimiter{\norm}{\|}{\|}
\DeclarePairedDelimiterX{\inprod}[2]{\langle}{\rangle}{#1,\ #2}

\newcommand*{\ball}[2]{\mathrm{B}(#1,#2)} 

\title{The relative $f$-invariant and non-uniform random sofic approximations}
\author{Christopher Shriver}

\begin{document}
\maketitle

\begin{abstract}
The $f$-invariant is an isomorphism invariant of free-group measure-preserving actions introduced by Lewis Bowen in \cite{bowen2010a}, where it was used to show that two finite-entropy Bernoulli shifts over a finitely generated free group can be isomorphic only if their base measures have the same Shannon entropy.
In \cite{bowen2010} Bowen showed that the $f$-invariant is a variant of sofic entropy; in particular it is the exponential growth rate of the expected number of good models over a uniform random homomorphism.

In this paper we present an analogous formula for the relative $f$-invariant and use it to prove a formula for the exponential growth rate of the expected number of good models over a random sofic approximation which is a type of stochastic block model.
\end{abstract}

\tableofcontents

\newpage

\section{Introduction, Main Results}

Let $G = \langle S \rangle$ denote the the rank-$r$ free group with generating set $S = \{s_1, \ldots, s_r\}$ and identity $e$, and let $(X,\mu,T)$ be a measure-preserving $G$-system, i.e. $T$ is a homomorphism from $G$ to the automorphism group of the standard probability space $(X,\mu)$. We will not need to make explicit use of the $\sigma$-algebra on $X$, so we leave it unnamed.

An \emph{observable} on $X$ is a measurable map with domain $X$. In this paper the codomain will be a finite set endowed with the discrete sigma algebra; in this case we call the map a \emph{finite observable} and the codomain an \emph{alphabet}.

Any observable $\alpha \colon X \to \A$ induces a map $\alpha^G \colon X \to \A^G$ by setting
	\[ (\alpha^G(x))_g = \alpha(T_g x) \quad \text{for all } g \in G . \]
The $\A$-coloring $\alpha^G(x)$ of $G$ is sometimes called the \emph{itinerary} of $x$, since it records the observations that will be made over the entire orbit of $x$ under the action of $G$. We also similarly define the map $\alpha^H \colon X \to \A^H$ for any subset $H$ of $G$. We abbreviate $\alpha^n \coloneqq \alpha^{\ball{e}{n}}$, where $\ball{e}{n}$ is the closed ball of radius $n$ centered at the identity in $G$, which is endowed with the word-length metric. If $\beta \colon X \to \B$ is a second finite observable, we denote by $\alpha\beta \colon X \to \A \times \B$ the map $\alpha\beta(x) = (\alpha(x), \beta(x))$.
\bigbreak

The (Shannon) entropy of a finite observable $\alpha \colon X \to \A$ is defined by
	\[ \shent_\mu (\alpha) = - \sum_{a \in \A} \alpha_*\mu (a) \log \alpha_* \mu(a) , \]
where $\alpha_* \mu \in \Prob(\A)$ is the pushforward measure; we take the convention $0 \log 0 = 0$. The entropy of $\alpha$ can be interpreted as the expected amount of information revealed by observing $\alpha$, assuming its distribution $\alpha_* \mu$ is known.

An early application of Shannon's entropy to ergodic theory was its use by Kolmogorov and Sinai to show that there exist nonisomorphic Bernoulli shifts over $\ZZ$. A Bernoulli shift over $\ZZ$ is a system of the form $(\A^\ZZ, \mu^\ZZ, S)$ for some alphabet $\A$ and $\mu \in \Prob(\A)$; $S$ is the shift action of $\ZZ$. They did this by defining an \emph{entropy rate} for $\ZZ$-systems, which can be interpreted as the average information per unit time revealed by observing the system. For a Bernoulli shift $(\A^\ZZ, \mu^\ZZ, S)$, the entropy rate is simply the ``base entropy'' $\shent_\mu(\alpha)$, where $\alpha \colon \A^n \to \A$ is the ``time zero'' observable.

Isomorphism invariance of the KS entropy rate is typically proven using the fact that entropy rate is nonincreasing under factor maps (which are surjective homomorphisms of measure-preserving systems). This fact can be interpreted as stating that a system cannot simulate another system that is ``more random.''

The entropy rate was soon generalized to systems acted on by an arbitrary amenable group (such as $\ZZ^d$). Extending beyond amenable groups proved more difficult, and in fact it was found to be impossible for such an extension to preserve all desirable properties of the KS entropy rate. In particular, an entropy rate for nonamenable group which assigns Bernoulli shifts their base entropy cannot be nonincreasing under factor maps \cite[Appendix C]{ornstein1987}.

\bigbreak

The first invariant to distinguish between Bernoulli shifts over free groups is Lewis Bowen's $\f$-invariant. Following \cite{bowen2010}, this can be defined by
\begin{align*}
	F_\mu (T, \alpha) &= (1-2r) \shent_\mu (\alpha) + \sum_{i=1}^r \shent_\mu (\alpha^{\{e,s_i\}}) \\
	\f_\mu (T, \alpha) &= \inf_n F_\mu (T, \alpha^n) = \lim_{n \to \infty} F_\mu (T, \alpha^n).
\end{align*}
The main theorem of \cite{bowen2010a} is that $\f_\mu(T, \alpha)$ depends on the observable $\alpha$ only through the $\sigma$-algebra it generates. In particular, the common value of $\f_\mu (T, \alpha)$ among all $\alpha$ which generate the Borel $\sigma$-algebra on $X$ (assuming such $\alpha$ exist) is a measure-conjugacy invariant of the system $(X, \mu, T)$. In the same paper, he showed that the $\f$-invariant of a Bernoulli shift is the Shannon entropy of the base measure; in particular, Bernoulli shifts with different base entropies are nonisomorphic.

\bigbreak

In \cite{bowen2010}, Bowen gave an alternate formula for the $f$-invariant, which we now introduce.

For any homomorphism $\sigma \colon G \to \Sym(n)$ we have a $G$-system $([n], \Unif(n), \sigma)$, and we can consider a labeling $\mb{x} \in \A^{n}$ as an observable on this system. We denote the law of its itinerary by $P^\sigma_{\mb{x}} = \mb{x}^G_*\Unif(n)$ and call this the \emph{empirical distribution} of $\mb{x}$. We say that $\mb{x}$ is a good model for $\alpha$ over $\sigma$ if it is difficult to distinguish the $G$-systems $(X,\mu, T)$ and $([n], \Unif(n), \sigma)$ via their respective observables $\alpha$ and $\mb{x}$. To make this precise, we denote
	\[ \Omega(\sigma, \calO) \coloneqq \{ \mb{x} \in \A^n \st P^\sigma_{\mb{x}} \in \calO \}, \]
which is a set of good models for $\alpha$ over $\sigma$ if $\calO$ is a weak$^*$-open neighborhood of $\alpha^G_*\mu \in \Prob(\A^G)$; the particular set $\calO$ quantifies how good the models are. The alphabet $\A$ is given the discrete topology and $\A^G$ the product topology, so ``weak$^*$-close'' means marginals on some finite sets are close in total variation norm.

For each $n \in \NN$, let  $\mu_n = \Unif(\Hom(G, \Sym(n)))$. Bowen showed in \cite{bowen2010} that the $f$-invariant is given by
	\[ f_\mu (T, \alpha) = \inf_{\calO \ni \alpha^G_*\mu} \limsup_{n \to \infty} \frac{1}{n} \log \EE_{\sigma \sim \mu_n} \abs{\Omega(\sigma, \calO)} . \]

To make an analogy with statistical physics, we can think of $\alpha^G_* \mu$ as a macroscopic statistical distribution of the state of a system; then the $f$-invariant is the exponential growth rate of the number of ``microstates'' that are consistent with these statistics. What we here call good models are often called microstates for this reason.

\bigbreak

If $\beta \colon X \to \B$ is a second observable, the conditional entropy is
	\[ \shent_\mu (\alpha | \beta) = \shent_\mu (\alpha \tup \beta) - \shent_\mu (\beta) . \]
This can be interpreted as the expected amount of information revealed by observing $\alpha$ if both the value of $\beta$ and the joint distribution of $\alpha$ and $\beta$ are known. By analogy we define
\begin{align*}
	F_\mu(T, \alpha | \beta) &= F_\mu (T, \alpha \tup \beta) - F_\mu (T, \beta) \\
	&= (1-2r) \shent_\mu (\alpha|\beta) + \sum_{i=1}^r \shent_\mu (\alpha^{\{e,s_i\}} \mid \beta^{\{e,s_i\}}) \\
	f_\mu (T, \alpha | \beta) &= \inf_{k_1 \in \NN} \sup_{k_2 \in \NN} F_\mu (T, \alpha^{k_1} \mid \beta^{k_2}).
\end{align*}
Both the infimum and supremum can be replaced by limits; this follows from Lemma \ref{lem:splittings} below. It follows from Corollary \ref{cor:chainrule} that we could also directly define
	\[ f_\mu (T, \alpha | \beta) = f_\mu(T, \alpha \tup \beta) - f_\mu(T, \beta) , \]
as long as $f_\mu(T, \beta) > -\infty$.

\bigbreak

A few more definitions are required to state our main theorems. If $H$ is a  finite subset of $G$, we denote by $d^H(\mu, \nu)$ the total variation distance between the marginals of $\mu$ and $\nu$ on $\A^H$. Our convention for the total variation distance between measures $\mu, \nu \in \Prob(\A)$ is
	\[ \norm{\mu - \nu}_{\TV} = \frac{1}{2} \sum_{a \in \A} \abs{\mu\{a\} - \nu\{a\}} . \]

For each $k \in \NN$ we define a pseudometric on $\Prob(\A^G)$ by
	\[ d^*_k(\mu, \nu) = \sum_{i \in [r]} d^{\ball{e}{k} \cup \ball{s_i}{k}}(\mu, \nu) . \]
Note that $\{d_k^*\}_{k \in \NN}$ together generate the weak$^*$ topology on $\Prob(\A^G)$. These generalize the almost-pseudometric\footnote{Bowen's $d_\sigma^*$ is essentially a pseudometric except that its first and second arguments come from different sets.} $d_\sigma^*$ from \cite{bowen2010}, which corresponds to the case $k=0$.
For $\calO = \{\nu \in \Prob(\A^G) \st d^*_k(\alpha^G_* \mu, \nu) < \varepsilon\}$ we write
	\[ \Omega(\sigma, \calO) \eqqcolon \Omega^*_k(\sigma, \alpha, \varepsilon) \subseteq \A^n . \]

In the present paper, instead of picking a homomorphism $\sigma$ uniformly at random we will use the following type of stochastic block model: given $\mb{y}_0 \in \B^n$, $\sigma_0 \in \Hom(G, \Sym(n))$, and $k \in \NN$, let
	\[ \SBM(\sigma_0, \mb{y}_0, k) \coloneqq \Unif(\{ \sigma \in \Hom(G, \Sym(n)) \st d^*_k(P_{\mb{y}_0}^{\sigma}, P_{\mb{y}_0}^{\sigma_0}) = 0 \}) . \]
The labeling $\mb{y}_0$ partitions the elements of $[n]$ into $\abs{\B}$ communities, and we can think of the random homomorphism $\sigma$ as a random choice of directed edges between and within the communities. Certain statistics of these random edge choices are determined by the reference homomorphism $\sigma_0$; note that for $k>0$ these statistics are more precise than those specified by a standard stochastic block model. In Section \ref{sec:weights} we define weights, which are the objects used to record the relevant statistics.


\bigbreak


We first prove our formula for the relative $f$-invariant under a Markov assumption: in this case, our stochastic block model only needs to take into account ``one-step statistics.''

\begin{mainthm}
\label{thm:main1}
	Let $\alpha \colon X \to \A$ and $\beta \colon X \to \B$ be finite observables, and for each $n$ let $\mb{y}_n \in \B^n$ and $\sigma_n \in \Hom(G, \Sym(n))$ be such that
		\[ \lim_{n \to \infty} d_0^*(P_{\mb{y}_n}^{\sigma_n}, \beta^G_*\mu) = 0 . \]
	Suppose that $\beta^G_* \mu$ is a Markov measure. With $\mu_n = \SBM(\sigma_n,\mb{y}_n, 0)$, we have
		\[ f_\mu(T,\alpha \mid \beta) = \inf_{\calO \ni (\alpha\beta)^G_* \mu} \limsupinf_{n \to \infty} \frac{1}{n} \log \EE_{\sigma \sim \mu_n} \abs*{\{ \mb{x} \in \A^n \st (\mb{x},\mb{y}_n) \in \Omega(\sigma, \calO) \}}. \]
\end{mainthm}
\begin{mainprop}
\label{prop:main1}
	The assumptions of Theorem \ref{thm:main1} are nonvacuous; that is, for any finite observable $\beta \colon X \to \B$ there exist sequences $\{\mb{y}_n \in \B^n\}_{n=1}^\infty$ and $\{ \sigma_n \in \Hom(G, \Sym(n)) \}_{n=1}^\infty$ such that $\lim_{n \to \infty} d_0^*(P_{\mb{y}_n}^{\sigma_n}, \beta^G_*\mu) = 0$.
\end{mainprop}


If $\beta^G_* \mu$ is not Markov, then the same formula holds with a more precise type of stochastic block model:

\begin{mainthm}
\label{thm:main2}
	Let $\alpha \colon X \to \A$ and $\beta \colon X \to \B$ be finite observables. Let $m_n$ approach infinity as $n$ goes to infinity while satisfying $m_n = o(\log \log n)$. For each $n$ let $\mb{y}_n \in \B^n$ and $\sigma_n \in \Hom(G, \Sym(n))$ be such that
		\[ d_{m_n}^*(P_{\mb{y}_n}^{\sigma_n}, \beta^G_*\mu) = O \big( \tfrac{1}{\log n} \big) . \]
	Suppose that $f_\mu(T, \beta) > -\infty$. With $\mu_n = \SBM(\sigma_n, \mb{y}_n, m_n)$,
		\[ f_\mu (T, \alpha \mid \beta) = \inf_{\calO \ni (\alpha\beta)^G_*\mu} \limsupinf_{n \to \infty} \frac{1}{n} \log \EE_{\sigma \sim \mu_n} \abs{\{ \mb{x} \in \A^n \st (\mb{x}, \mb{y}_n) \in \Omega(\sigma, \calO) \}} . \]
\end{mainthm}
\begin{mainprop}
\label{prop:main2}
	The assumptions of Theorem \ref{thm:main2} are nonvacuous; that is, for any finite observable $\beta \colon X \to \B$ and any sequence $\{m_n \in \NN\}_{n=1}^\infty$ approaching infinity while satisfying $m_n = o(\log \log n)$, there exist sequences $\{\mb{y}_n \in \B^n\}_{n=1}^\infty$ and $\{ \sigma_n \in \Hom(G, \Sym(n)) \}_{n=1}^\infty$ such that $\lim_{n \to \infty} d_{m_n}^*(P_{\mb{y}_n}^{\sigma_n}, \beta^G_*\mu) = O \left( \frac{1}{\log n} \right)$.
\end{mainprop}

The expressions appearing on the right-hand sides of Theorems \ref{thm:main1} and \ref{thm:main2} are very closely related to Ben Hayes' definition of ``relative sofic entropy in the presence'' \cite[Definition 2.5]{hayes2016}. Some differences are that we consider expected numbers of good models over random sofic approximations, and that Hayes takes a supremum inside the logarithm over which good model is to be extended, while we fix a sequence $\{\mb{y}_n\}$ of planted good models. Hayes also does not restrict to shift systems as we do here.

\bigbreak

Using Theorem \ref{thm:main2} we prove the following formula for the growth rate of the expected number of good models over a homomorphism drawn from a stochastic block model:
	
\begin{mainthm}
\label{thm:main3}
	Let $\mu_n, \alpha, \beta$ be as in the statement of Theorem \ref{thm:main2}. Then
		\[ \inf_{\calO \ni \alpha^G_*\mu} \limsup_{n \to \infty} \frac{1}{n} \log \EE_{\sigma \sim \mu_n} \abs{\Omega(\sigma, \calO)} = \sup_{\lambda \in \join(\alpha_*^G \mu,\, \beta_*^G \mu)} f_\lambda (S, \amap \mid \bmap) . \]
\end{mainthm}
Here $\join(\alpha_*^G \mu,\, \beta_*^G \mu)$ is the set of joinings of the $G$-systems $(\A^G, \alpha^G_*\mu, S)$ and $(\B^G, \beta^G_*\mu, S)$, i.e. shift-invariant probability measures on $(\A \times \B)^G$ whose $\A^G, \B^G$ marginals are $\alpha_*^G \mu,\, \beta_*^G \mu$, respectively.
$S$ denotes the shift action of $G$. We use $\amap, \bmap$ to denote the maps
\begin{align*}
	\amap \colon (\A\times \B)^G &\to \A			&	\bmap \colon (\A \times \B)^G &\to \B \\
	\big( (a_g,b_g) \big)_{g \in G} &\mapsto a_e	&	\big( (a_g,b_g) \big)_{g \in G} &\mapsto b_e 
\end{align*}
which observe the $\A$ (resp. $\B$) label at the identity.

Note: the supremum is always greater than or equal to $f_\mu (T, \alpha)$, with equality attained by the product joining; this means that the expected number of good models for $\alpha$ over a block model with built-in good models for any $\beta$ is at least the expected number of good models over a uniformly random homomorphism. It is possible for the supremum to be strictly larger, however. For example, suppose $f_\mu (T, \alpha) < 0$ and $\alpha = \beta$, and let  $\lambda$ be the diagonal joining. Then
	\[ f_{\lambda}(S, \amap \mid \bmap) = 0 > f_\mu (T, \alpha) . \]

\subsection{Random sofic approximations}

The $f$-invariant is closely related to another invariant of measure-preserving systems called sofic entropy, which was introduced by Lewis Bowen in \cite{bowen2010b}.

A homomorphism $\sigma \in \Hom(G, \Sym(n))$ is called $(D,\delta)$-sofic for some finite $D \subset G$ and $\delta > 0$ if 
	\[ \abs*{\{ j \in [n] \st \sigma(\gamma) j \neq j \ \forall \gamma \in D \setminus\{e\} \}} > (1-\delta) n . \]
A sequence of homomorphisms $\Sigma = \big(\sigma_n \in \Hom(G, \Sym(n)) \big)_{n \in \NN}$ is called a sofic approximation if for every $(D,\delta)$ the homomorphism $\sigma_n$ is $(D,\delta)$-sofic for all large enough $n$.

The sofic entropy relative to $\Sigma$ is the exponential growth rate of the number of good models over $\sigma_n$. Specifically, if there is some finite observable $\alpha$ which generates the Borel $\sigma$-algebra on $X$ then we have
	\[ \h_\Sigma(\mu, T) = \inf_{\calO \ni \alpha_*^G \mu} \limsup_{n \to \infty} \frac{1}{n} \log \abs*{\Omega (\sigma_n, \calO)} . \]
	
By analogy with this expression, we might call the sequences of random homomophisms appearing in expressions above ``random sofic approximations.'' The following proposition provides further justification for this terminology.

\begin{prop}
\label{prop:sofic}
	If $(\mu_n)$ is any of the sequences appearing in Theorems \ref{thm:main1}, \ref{thm:main2}, and \ref{thm:main3}, then for any $(D,\delta)$ there exists $\varepsilon>0$ such that
		\[ \PP_{\sigma \sim \mu_n} \big( \sigma \text{ is } (D,\delta) \text{-sofic} \big) \geq 1 - n^{-\varepsilon n} \]
	for all large enough $n$.
\end{prop}

In particular, if $\sigma_1 \sim \mu_1$, $\sigma_2 \sim \mu_2$ etc.~are independent then $(\sigma_n)$ is a sofic approximation with probability 1.

\subsection*{Organization}
In Section \ref{sec:weights} we define weights and discuss some of their useful properties. In Section \ref{sec:Fandf} we prove a few basic results about the functions $f$ and $F$. Some of the results of these two sections are used in Section \ref{sec:nonvacuity} to show that the assumptions of the main theorems are not vacuous. In Section \ref{sec:counting} we show how the function $F$ is related to the number of homomorphism-labeling pairs $(\sigma, \mb{y})$ that realize a given weight, which is the main ingredient of the proofs of Theorems \ref{thm:main1} and \ref{thm:main2} given in the next two sections. In Section \ref{sec:main3pf} we show how to deduce Theorem \ref{thm:main3} from Theorem \ref{thm:main2}. Section \ref{sec:soficpf} contains a proof of Proposition \ref{prop:sofic}. The final section contains a proof of Lemma \ref{lem:denomnapprox}, which asserts that a weight can be approximated by a denominator-$n$ weight with a specified marginal.

\subsection*{Acknowledgements}

Thanks to Tim Austin for suggesting that theorems similar to \ref{thm:main2} and \ref{thm:main3} should hold, for many helpful discussions, and for providing comments on an earlier draft. Thanks also to Ben Hayes for sharing helpful references.

This material is based upon work supported by the National Science Foundation under Grant No.~DMS-1855694.

\section{Weights}
\label{sec:weights}

If $\alpha \colon X \to \A$ is a finite observable, for $a,a' \in \A$ and $i \in [r]$ let
	\[ W_{\alpha}(a,a'; i) = \alpha_*^{\{e,s_i\}} \mu(a,a') = \mu\{ x \in X \st \alpha(x) = a,\ \alpha(T_{s_i} x) = a'\} \]
and also denote
	\[ W_{\alpha}(a) = \alpha_*\mu(a) . \]
For $\mb{x} \in \A^n$ and $\sigma \in \Hom(G, \Sym(n))$ let
	\[ W_{\sigma, \mb{x}}(a,a'; i) = P^{\sigma, \{e,s_i\}}_{\mb{x}} (a,a') \]
and $W_{\sigma, \mb{x}}(a) = P^{\sigma, \{e\}}_{\mb{x}} (a)$.

More abstractly, any $W \in \big( \Prob(\A^2) \big)^r$ is called an \emph{$\A$-weight} if 
	\[ \sum_{a' \in \A} W(a, a'; i) = \sum_{a' \in \A} W(a', a; j) \]
for all $i,j \in [r]$ and $a \in \A$. For each $a \in \A$ we denote this common value $W(a)$. Note that the objects $W_\alpha$ and $W_{\sigma, \mb{x}}$ defined above satisfy this condition.

We say that $W$ has denominator $n$ if $n \cdot W(a,a';i) \in \NN$ for all $a,a',i$.

The measures $W(\cdot,\cdot; i)$ for $i \in [r]$ are called the \emph{edge measures} of $W$, and $W(\cdot)$ is called the \emph{vertex measure}.

For any alphabet $\A$, we use the metric on $\A$-weights defined by
\begin{align*}
	d(W_1, W_2) &\coloneqq \sum_{i \in [r]} \norm*{ W_1(\cdot, \cdot; i) - W_2(\cdot, \cdot; i)}_{\TV} \\
	&= \frac{1}{2} \sum_{i \in [r]} \sum_{a,a' \in \A} \abs{ W_1(a,a'; i) - W_2(a,a'; i)} .
\end{align*}

We can use weights to count good models up to equivalence under the pseudometrics $d^*_k$ using the following proposition:
\begin{prop}
\label{prop:dstar}
If $\sigma \in \Hom(G, \Sym(n))$ and $\mb{x} \in \A^n$, then for any observable $\alpha \colon X \to \A$
	\[ d(W_{\sigma, \mb{x}^k}, W_{\alpha^k}) = d^*_k(P_{\mb{x}}^\sigma, \alpha^G_* \mu) . \]
\end{prop}
Note this implies also that
	\[ d^*_k(P_{\mb{x}}^\sigma, \alpha^G_* \mu) = d^*_0(P_{\mb{x}^k}^\sigma, (\alpha^k)^G_* \mu) . \]
\begin{proof}
	By definition of the distance between weights,
	\begin{align*}
		d(W_{\sigma, \mb{x}^k}, W_{\alpha^k}) &= \frac{1}{2} \sum_{i \in [r]} \sum_{\mb{a}, \mb{a}' \in \A^{\ball{e}{k}}} \abs*{W_{\sigma, \mb{x}^k}( \mb{a}, \mb{a}'; i) - W_{\alpha^k}(\mb{a}, \mb{a}'; i)} \\
		&= \frac{1}{2} \sum_{i \in [r]} \sum_{\mb{a}, \mb{a}' \in \A^{\ball{e}{k}}}
			\Bigg\lvert\,
				\frac{1}{n} \abs*{\left\{ j \in [n] \st \begin{array}{c}
											(\mb{x}^k)_j = \mb{a} \\
											(\mb{x}^k)_{\sigma(s_i)j} = \mb{a}'
										\end{array} \right\}} \\
		&\phantom{= \sum_{i \in [r]} \sum_{\mb{a}, \mb{a}' \in \A^{\ball{e}{k}}}}
				\qquad - \mu \left\{ x \in X \st \begin{array}{c}
											\alpha^k(x) = \mb{a} \\
											\alpha^k(T_{s_i}x) = \mb{a}'
										\end{array} \right\}
			\Bigg\rvert .
	\end{align*}
	For many `incompatible' pairs $\mb{a},\mb{a}'$, both terms will be zero: suppose $g \in \ball{e}{k} \cap \ball{s_i}{k}$, so that $g s_i^{-1} \in \ball{e}{k}$. If the second term in the absolute value is nonzero, then for some $x \in X$ we have $\alpha^k(x) = \mb{a}$ and $\alpha^k(T_{s_i}x) = \mb{a}'$, and therefore
		\[ \mb{a}'_{g s_i^{-1}} = (\alpha^k(T_{s_i}x))_{g s_i^{-1}} = \alpha(T_{g s_i^{-1}} T_{s_i} x) = \alpha(T_g x) = (\alpha^k (x) )_g = \mb{a}_g . \]
	The same argument shows that $\mb{a}'_{g s_i^{-1}} = \mb{a}_g$ for all $g \in \ball{e}{k} \cap \ball{s_i}{k}$ whenever the first term is nonzero. Therefore we can restrict the sum to pairs $\mb{a}, \mb{a}'$ with $\mb{a}'_{g s_i^{-1}} = \mb{a}_g$ for all $g \in \ball{e}{k} \cap \ball{s_i}{k}$. Equivalently, we can sum over all $\mb{A} \in \A^{\ball{e}{k} \cup \ball{s_i}{k}}$ to get
	\begin{align*}
		d(W_{\sigma, \mb{x}^k}, W_{\alpha^k}) &= \frac{1}{2} \sum_{i \in [r]} \sum_{\mb{A} \in \A^{\ball{e}{k} \cup \ball{s_i}{k}}}
			\Bigg\lvert
				\frac{1}{n} \abs*{\left\{ j \in [n] \st \big( \mb{x}^{\ball{e}{k} \cup \ball{s_i}{k}} \big)_j = \mb{A} \right\}} \\
		&\phantom{= \sum_{i \in [r]} \sum_{\mb{a}, \mb{a}' \in \A^{\ball{e}{k}}}}
				\qquad - \mu \left\{ x \in X \st \alpha^{\ball{e}{k} \cup \ball{s_i}{k}} (x) = \mb{A} \right\}
			\Bigg\rvert \\
		&= \sum_{i \in [r]} d^{\ball{e}{k} \cup \ball{s_i}{k}} ( P_{\mb{x}}^\sigma, \alpha_*^G \mu) . \qedhere
	\end{align*}	
\end{proof}

\bigbreak

It will be useful to consider the pushforward map induced by a map between alphabets: if $\pi \colon \A \to \B$ is a measurable map and $W$ is an $\A$-weight, then $\pi W$ is the $\B$-weight given by 
	\[ \pi W(b, b'; i) = \sum_{a \in \pi^{-1}\{b\}} \sum_{a' \in \pi^{-1}\{b'\}} W(a,a'; i) . \]
Note that this implies that the vertex measure of $W$ is
	\[ \pi W(b) = \sum_{a \in \pi^{-1}\{b\}} W(a) . \]
	
For example, let $\pi_\B \colon \A \times \B \to \B$ be the projection map. If $W$ is an $\A \times \B$-weight then $\pi_\B W$ is given by
	\[ \pi_\B W(b_1) = \sum_{a \in \A} W\big((a,b_1)\big) \qquad \pi_\B W(b_1, b_2; i) = \sum_{a_1, a_2 \in \A} W \big((a_1, b_1), (a_2, b_2); i \big) . \]
We call this the $\B$-marginal of $W$.

All weights in the present paper will be over alphabets of the form $\A^{\ball{e}{k}} \times \B^{\ball{e}{k'}}$. We use this fact to introduce some simplified notation for projections:

\begin{itemize}
\item $\pi_A$ denotes projection onto the entire $\A$ factor $\A^{\ball{e}{k}}$; $\pi_B$ is used similarly.

\item For $m<k$ and $m'<k'$, $\pi_{m,m'}$ denotes projection onto $\A^{\ball{e}{m}} \times \B^{\ball{e}{m'}}$.

\item $\pi_m$ denotes the projection $\A^{\ball{e}{k}} \to \A^{\ball{e}{m}}$, except that if $m=0$ we write $\pi_e$.
\end{itemize}

\bigbreak

We define $F(W)$ for an abstract weight $W$ by
	\[ F(W) = (1-2r) \shent \big( W(\cdot) \big) + \sum_{i \in [r]} \shent \big( W(\cdot, \cdot; i) \big) \]
where $H$ is the Shannon entropy. Note that this is consistent with the above definitions in that, for example,
	\[ F(W_\alpha) = F_\mu (T, \alpha) . \]
	
\bigbreak

We can revisit the definition of our version of the stochastic block model using weights: Let $H \subset G$ and let $W$ be a denominator-$n$ $\B^{\ball{e}{k}}$-weight. Suppose there exist $\mb{y} \in \B^n$ and $\sigma \in \Hom(G, \Sym(n))$ such that $W = W_{\sigma, \mb{y}^k}$. Then
	\[ \SBM(\sigma, \mb{y}, k) = \Unif(\{ \sigma' \in \Hom(G, \Sym(n)) \st W_{\sigma', \mb{y}^k} = W \}) , \]
so we can also denote this distribution by $\SBM(\mb{y}, W)$. Specifying the distribution by a weight rather than a specific homomorphism will occasionally be more convenient.

\subsection{Constructing Weights and Good Models}

We borrow the first result of this type from \cite{bowen2010}; it allows us to find a denominator-$n$ approximation to a given weight.

\begin{lemma}[Lemma 2.3 of \cite{bowen2010}]
\label{lem:bowenweightapprox}
	There is a constant $C$ such that for any $\A$-weight $W$ there is a denominator-$n$ $\A$-weight within distance $C \abs{\A}^2 r/n$ of $W$.
\end{lemma}

The following lemma allows us not only to construct a denominator-$n$ approximation to a given weight, but also to specify a marginal of this approximation:
\begin{lemma}
\label{lem:denomnapprox}
	Let $W$ be an $\A \times \B$-weight. If $W_\B$ is a $\B$-weight of denominator $n$ with $d(W_\B, \pi_\B W) < \delta$ then there is an $\A \times \B$-weight $W_{\A\B}$ with denominator $n$ such that $\pi_\B W_{\A\B} = W_\B$ and $d(W_{\A\B}, W) < 265r ( \abs{\A \times \B}^2  / n + \delta)$.
\end{lemma}

The construction is fairly involved, so is postponed to Section \ref{sec:denomnapproxproof}. The constant 265 is not intended to be optimal.

The definition of a weight $W_{\sigma, \mb{x}^k}$ in terms of a homomorphism $\sigma$ and a labeling $\mb{x}$ is straightforward. However, we will also need to know whether a given weight can be realized in this way. The next two results address this inverse problem.

\begin{prop}
\label{prop:weightinverse}
	If $W$ is a denominator-$n$ $\A$-weight, then there exist $\mb{x} \in \A^n$ and $\sigma \in \Hom(G, \Sym(n))$ such that $W = W_{\sigma, \mb{x}}$.
\end{prop}
\begin{proof}
	This is implied by Proposition 2.1 of \cite{bowen2010}.
\end{proof}

Unfortunately, this does not imply that for every denominator-$n$ $\A^{\ball{e}{k}}$-weight $W$ there is some $\sigma \in \Hom(G, \Sym(n))$ and $\mb{x} \in \A^n$ such that $W = W_{\sigma, \mb{x}^k}$; instead it provides $\mb{X} \in (\A^{\ball{e}{k}})^n$ such that $W = W_{\sigma, \mb{X}}$.

However, if we already know that $W$ is close to a weight of the form $W_{\alpha^k}$ for some observable $\alpha$, then the following proposition shows that $W$ is also close to a weight of the form $W_{\sigma, \mb{x}^k}$.

\begin{prop}
\label{prop:weightballapprox}
	Let $\alpha \colon X \to \A$, $\sigma \in \Hom(G, \Sym(n))$, and $\mb{X} \in (\A^{\ball{e}{k}})^n$ be such that $d(W_{\sigma, \mb{X}}, W_{\alpha^k}) \leq \varepsilon$ for some $\varepsilon \geq 0$. Writing $\mb{x} = \pi_e \mb{X} \in \A^n$, we have
		\[ d(W_{\sigma, \mb{X}}, W_{\sigma, \mb{x}^k}) \leq 2r \abs{\ball{e}{k}} \varepsilon . \]
\end{prop}
An immediate consequence is that $\mb{X} \in \Omega_0^*(\sigma, \alpha^k, \varepsilon)$ implies $\pi_e \mb{X} \in \Omega_k^*(\sigma, \alpha, c \varepsilon)$ where $c = 1 + 2r \abs{\ball{e}{k}}$; cf. Claim 2 in the proof of Proposition 3.2 of \cite{bowen2010}.

\begin{proof}
	Claim 4 in the proof of Proposition 3.2 of \cite{bowen2010} implies that
		\[ \abs{\{ j \in [n] \st \mb{X}(j) \ne \mb{x}^k(j) \}} \leq n \abs{\ball{e}{k}} \varepsilon . \]
	It follows that for any $i \in [r]$
	\begin{align*}
		&\hspace{-2em} \abs{\{ j \in [n] \st \mb{X}^{\{e, s_i\}}(j) \ne (\mb{x}^k)^{\{e,s_i\}}(j) \}} \\
		&\leq \abs{\{ j \in [n] \st \mb{X}(j) \ne \mb{x}^k(j) \}} + \abs{\{ j \in [n] \st \mb{X}(\sigma(s_i)j) \ne \mb{x}^k(\sigma(s_i)j) \}} \\
		&\leq 2 n \abs{\ball{e}{k}} \varepsilon ,
	\end{align*}
	so
	\begin{align*}
		d(W_{\sigma, \mb{X}}, W_{\sigma, \mb{x}^k}) &= \sum_{i \in [r]} \norm*{ \left(\mb{X}^{\{e, s_i\}}\right)_* \Unif(n) - \left((\mb{x}^k)^{\{e, s_i\}}\right)_* \Unif(n)}_{\TV} \\
		&\leq \sum_{i \in [r]} 2 \abs{\ball{e}{k}} \varepsilon = 2 r \abs{\ball{e}{k}} \varepsilon . \qedhere
	\end{align*}
\end{proof}

\section{Properties of $F$ and $f$}
\label{sec:Fandf}

\begin{lemma}[Continuity as weight function]
\label{lem:Fcontinuity}
	If $W_1, W_2$ are $\A$-weights with $d(W_1, W_2) \leq \varepsilon 
	\leq 1$ then
		\[ \abs{F(W_1) - F(W_2)} \leq 4r \big( \shent(\varepsilon) + \varepsilon \log_2 \abs{\A} \big) . \]
	where $\shent(p)$ denotes the entropy of the probability measure $(p, 1-p) \in \Prob(\{0,1\})$.
\end{lemma}
\begin{proof}
	We use Fano's inequality in the following form (Equation (2.139) of \cite{cover2006a}): suppose $X,Y$ are $\A$-valued random variables defined on the same probability space and let $p_e = \PP(X \ne Y)$ be their probability of disagreement. Then
		\[ \shent(X \mid Y) \leq \shent(p_e) + p_e \log \abs{\A} . \]
	Using the chain rule and nonnegativity of Shannon entropy, we can deduce that
		\[ \abs{ \shent(X) - \shent(Y)} \leq \shent(p_e) + p_e \log \abs{\A}. \]
	Let $\mu_1, \mu_2 \in \Prob(\A)$ be the respective distributions of $X_1,X_2$. Because $\norm{\mu_1 - \mu_2}_{\TV}$ is the minimum value of $\PP(X \ne Y)$ over all possible couplings, if $\norm{\mu_1 - \mu_2}_{\TV} < \varepsilon$ then
		\[ \abs{\shent(\mu_1) - \shent(\mu_2)} \leq \shent(\varepsilon) + \varepsilon \log_2 \abs{\A} . \]
	
	The assumed bound $d(W_1, W_2) \leq \varepsilon$ implies that each vertex and edge measure of $W_1$ is within total variation distance $\varepsilon$ of its counterpart in $W_2$, so
	\begin{align*}
		\abs{F(W_1) - F(W_2)} &\leq \abs{1-2r} \cdot \abs*{\shent \big( W_1(\cdot) \big) - \shent \big( W_2(\cdot) \big)} \\
		&\qquad + \sum_{i \in [r]} \abs*{\shent \big( W_1(\cdot, \cdot; i) \big) - \shent \big( W_2(\cdot, \cdot; i) \big)} \\
		&\leq (2r-1) \left( \shent(\varepsilon) + \varepsilon \log_2 \abs{\A} \right) \\
		&\qquad + r \cdot \left( \shent(\varepsilon) + \varepsilon \log_2 \abs{\A}^2 \right) \\
		&\leq 4r \big( \shent(\varepsilon) + \varepsilon \log_2 \abs{\A} \big). \qedhere
	\end{align*}
\end{proof}

Let $\alpha \colon X \to \A$ and $\beta \colon X \to \B$ be observables. We say that $\beta$ is a \emph{coarsening} of $\alpha$ if each part of the partition of $X$ induced by $\beta$ is a union of parts of the partition induced by $\alpha$ (up to null sets). Equivalently, there is some function $g \colon \A \to \B$ such that $\beta = g \circ \alpha$ almost surely. In this situation we can also call $\alpha$ a refinement of $\beta$.

A useful property of the Shannon entropy $\shent_\mu(\alpha)$ is monotonicity under refinement. The function $F$ does not share this property, but it is monotone under the following particular kind of refinement introduced in \cite{bowen2010a}:

We say that $\beta$ is a \emph{simple splitting} of $\alpha$ if there is some $s \in \{s_1^{\pm1}, \ldots, s_r^{\pm1}\}$ and a coarsening $\tilde{\alpha}$ of $\alpha$ such that, up to null sets, the partition induced by $\beta$ is the coarsest common refinement of the partitions induced by $\alpha$ and $\tilde{\alpha} \circ T_s$.

We say that $\beta$ is a \emph{splitting} of $\alpha$ if there are observables $\alpha = \beta_0, \beta_1, \ldots, \beta_n = \beta$ such that $\beta_i$ is a simple splitting of $\beta_{i-1}$ for $i = 1, 2, \ldots, n$. We will use the following monotonicity properties of the relative version of $F$:

\begin{lemma}[Monotonicity under splitting]
\label{lem:splittings}
	\ 
	\begin{enumerate}
		\item If $\alpha_1$ is a splitting of $\alpha_2$ then $F(\alpha_1 | \beta) \leq F(\alpha_2 | \beta)$.
		\item If $\beta_1$ is a splitting of $\beta_2$ then $F(\alpha | \beta_1) \geq F(\alpha | \beta_2)$.
	\end{enumerate}
\end{lemma}
\begin{proof}
	\begin{enumerate}
		\item This is essentially Proposition 5.1 of \cite{bowen2010a}; conditioning on $\beta$ makes no difference to the proof.
		
		\item The proof is based on the proof of Part 1, but in place of the chain rule for conditional entropy we use the following bound:
		\begin{align*}
			\shent(\alpha \mid \beta_2)
				&\leq \shent(\alpha, \beta_1 \mid \beta_2) & \text{(monotonicity)} \\
				&= \shent(\beta_1 \mid \beta_2) + \shent(\alpha \mid \beta_1, \beta_2) & \text{(chain rule)} \\
				&\leq \shent(\beta_1 \mid \beta_2) + \shent(\alpha \mid \beta_1) & \text{(monotonicity)} \mathrlap{.}
		\end{align*}
		We will also use the following consequence of the previous bound:
		\begin{align*}
			&\hspace{-3em}\shent(\alpha^{\{e,s_i\}} \mid \beta_1^{\{e,s_i\}}) - \shent(\alpha^{\{e,s_i\}} \mid \beta_2^{\{e,s_i\}}) \\
			&\geq -\shent(\beta_1^{\{e,s_i\}} \mid \beta_2^{\{e,s_i\}}) & \text{(previous bound)} \\
			&\geq -\big( \shent(\beta_1^{\{s_i\}} \mid \beta_2^{\{e,s_i\}}) + \shent(\beta_1 \mid \beta_2^{\{e,s_i\}}) \big) & \text{(subadditivity)} \\
			&= -\big( \shent(\beta_1 \mid \beta_2^{\{e,s_i^{-1}\}}) + \shent(\beta_1 \mid \beta_2^{\{e,s_i\}}) \big) & \text{($T$-invariance of $\mu$)} \mathrlap{.}
		\end{align*}
		
		It suffices to check the case where $\beta_1$ is a simple splitting of $\beta_2$: let $t \in \{s_1^{\pm 1}, \ldots, s_r^{\pm 1} \}$ and let $\tilde{\beta}$ be a coarsening of $\beta_2$ such that the partition induced by $\beta_1$ is the same as the coarsest common refinement of the partitions induced by $\beta_2$ and $\tilde{\beta} \circ T_t$ up to null sets. Then, using the two bounds just derived,
		\begin{align*}
			F(\alpha | \beta_1) - F(\alpha | \beta_2) 
			&= (1-2r) \left( \shent(\alpha | \beta_1) - \shent(\alpha | \beta_2) \right) \\
			&\qquad + \sum_{i \in [r]} \left( \shent(\alpha^{\{e,s_i\}}  | \beta_1^{\{e,s_i\}}) - \shent(\alpha^{\{e,s_i\}}  | \beta_1^{\{e,s_i\}}) \right) \\
			&\geq (1-2r) \left( - \shent(\beta_1 | \beta_2) \right) - \sum_{i \in [r]} \left( \shent(\beta_1 \mid \beta_2^{\{e,s_i^{-1}\}}) + \shent(\beta_1 \mid \beta_2^{\{e,s_i\}}) \right) \\
			&= (2r-1) \shent(\beta_1 | \beta_2) - \sum_{s \in \{s_1^{\pm 1} \ldots s_r^{\pm 1}\}} \shent(\beta_1 \mid \beta_2^{\{e,s\}})
		\end{align*}
		But
			\[ \shent(\beta_1 \mid \beta_2^{\{e,t\}}) \leq \shent(\beta_1 \mid \beta_2 \tilde{\beta}^{\{t\}}) = 0 , \]
		so we can remove the $t$ term from the sum to get
		\begin{align*}
			F(\alpha | \beta_1) - F(\alpha | \beta_2) &\geq (2r-1) \shent(\beta_1 | \beta_2) - \sum_{s \in \{s_1^{\pm 1} \ldots s_r^{\pm 1}\} \setminus \{t\}} \shent(\beta_1 \mid \beta_2^{\{e,s\}}) \\
			&= \sum_{s \in \{s_1^{\pm 1} \ldots s_r^{\pm 1}\} \setminus \{t\}} \left( \shent(\beta_1 | \beta_2) - \shent(\beta_1 \mid \beta_2^{\{e,s\}}) \right) \\
			&\geq 0 . \qedhere
		\end{align*}
	\end{enumerate}
\end{proof}

One corollary is the following convenient formula:
\begin{cor}
\label{cor:markovF}
	Let $\alpha, \beta$ be finite observables such that $\beta^G_*\mu$ is a Markov measure. Then $F_\mu(T, \alpha^{k_1} \mid \beta^{k_2})$ is independent of $k_2$. In particular, 
		\[ f_\mu(T, \alpha \mid \beta) = \inf_k F_\mu (T, \alpha^k \mid \beta) . \]
\end{cor}
\begin{proof}
	By the previous proposition, for any $k \leq k_2$ we have
		\[ F_\mu(T, \alpha^{k_1} \mid \beta^k) \leq F_\mu(T, \alpha^{k_1} \mid \beta^{k_2}) . \]
		
	On the other hand, by Theorem 6.1 of \cite{bowen2010c} $F_\mu(T, \beta^k) = F_\mu (T, \beta^{k_2})$ so
		\[ F_\mu(T, \alpha^{k_1} \mid \beta^{k})  = F_\mu(T, \alpha^{k_1} \beta^{k}) - F_\mu(T, \beta^{k_2}) . \]
	Applying monotonicity under splitting to the first term on the right gives
		\[ F_\mu(T, \alpha^{k_1} \mid \beta^{k})  \geq F_\mu(T, \alpha^{k_1} \beta^{k_2}) - F_\mu(T, \beta^{k_2}) = F_\mu(T, \alpha^{k_1} \mid \beta^{k_2}). \]
	
	This establishes independence of $k_2$; the formula for $f$ follows.
\end{proof}

\begin{prop}
\label{prop:relfbound}
	Let $\alpha, \beta$ be finite observables. Then for any $k \in \NN$,
		\[ F_\mu(T, \alpha^k \mid \beta ) \leq \shent_\mu \big( \alpha \mid \beta \big) . \]
	It follows that
		\[ f_\mu(T, \alpha \mid \beta ) \leq \shent_\mu \big( \alpha \mid \beta \big) . \]
\end{prop}
\begin{proof}
	By Lemma \ref{lem:splittings}, $F_\mu(T, \alpha^{k} \mid \beta ) \leq F_\mu(T, \alpha \mid \beta )$. Using elementary properties of Shannon entropy, we have
	\begin{align*}
		F_\mu(T, \alpha \mid \beta ) &= (1-2r) \shent_\mu (\alpha \mid \beta) + \sum_{i \in [r]} \shent_\mu \big( \alpha^{\{e,s_i\}} \mid \beta^{\{e, s_i\}} \big) \\
		&\leq (1-2r) \shent_\mu (\alpha \mid \beta) + \sum_{i \in [r]} \left[ \shent_\mu \big( \alpha \mid \beta^{\{e, s_i\}} \big) + \shent_\mu \big( \alpha^{\{s_i\}} \mid \beta^{\{e, s_i\}} \big) \right] \\
		&\leq (1-2r) \shent_\mu (\alpha \mid \beta) + \sum_{i \in [r]} \left[ \shent_\mu \big( \alpha \mid \beta \big) + \shent_\mu \big( \alpha^{\{s_i\}} \mid \beta^{\{s_i\}} \big) \right] .
	\end{align*}
	By $T$-invariance of $\mu$ we have
		\[ \shent_\mu \big( \alpha^{\{s_i\}} \mid \beta^{\{s_i\}} \big) = \shent_\mu ( \alpha \mid \beta ) , \]
	so the first inequality follows.
	
	For any $k_1, k_2 \in \NN$ this gives
		\[ F_\mu (T, \alpha^{k_1} \mid \beta^{k_2}) \leq \shent_\mu ( \alpha \mid \beta^{k_2} ) \leq \shent_\mu (\alpha \mid \beta) , \]
	so the second inequality follows upon taking the supremum over $k_2$ then the infimum over $k_1$.
\end{proof}

We can use this bound to give a proof of the chain rule for the relative $f$-invariant, a version of which first appeared in \cite{bowen2010c} (there it is called the Abramov-Rokhlin formula; see also \cite{bowen2013}):

\begin{cor}[Chain rule]
\label{cor:chainrule}
	\[ f_\mu (T, \alpha\tup\beta) = f_\mu (T, \alpha \mid \beta) + f_\mu (T, \beta). \]
\end{cor}

\begin{proof}
	By definition of the relative version of $F$ and the chain rule for conditional entropy, for each $k_1, k_2$ we have
		\[ F_\mu(T, \alpha^{k_1} \tup \beta^{k_2}) = F_\mu(T, \alpha^{k_1} \mid \beta^{k_2}) + F_\mu(T, \beta^{k_2}) . \]
	By Lemma \ref{lem:splittings} each term is monotone in $k_2$, so the limits as $k_2 \to \infty$ exist. By Proposition \ref{prop:relfbound} all terms are bounded above (recall we only consider finite observables, so in particular all observables have finite entropy), so we can split the limit across the sum on the right to get
		\[ \lim_{k_2 \to \infty} F_\mu(T, \alpha^{k_1} \tup \beta^{k_2}) =  \lim_{k_2 \to \infty} F_\mu(T, \alpha^{k_1} \mid \beta^{k_2}) + f_\mu(T, \beta) . \]
	Taking $k_1$ to infinity gives the result.
\end{proof}

\section{Non-vacuity of Main Theorems}
\label{sec:nonvacuity}

\subsection{Theorem \ref{thm:main1}}

Here we prove Proposition \ref{prop:main1}, which asserts the nonvacuity of Theorem \ref{thm:main1}. Given $\beta \colon X \to \B$, we need to show that there exist $\mb{y}_n \in \B^n$ and $\sigma_n \in \Hom(G, \Sym(n))$ such that $\lim_{n \to \infty} d_0^*(P_{\mb{y}_n}^{\sigma_n}, \beta^G_* \mu) = 0$.

By Lemma \ref{lem:bowenweightapprox}, there is a sequence $\{W_n\}_{n=1}^\infty$ of $\B$-weights such that $W_n$ has denominator $n$ for each $n$ and $d(W_n, W_\beta) = o(1)$. By Proposition \ref{prop:weightinverse}, for each $n$ we can pick $\mb{y}_n,\sigma_n$ such that $W_{\sigma_n, \mb{y}_n} = W_n$. Since $d_0^*(P_{\mb{y}_n}^{\sigma_n}, \beta^G_* \mu) = d(W_{\sigma_n, \mb{y}_n}, W_\beta)$, these suffice.

\subsection{Theorems \ref{thm:main2} and \ref{thm:main3}}

Here we prove Proposition \ref{prop:main2}, which asserts the nonvacuity of Theorem \ref{thm:main2} (and by extension Theorem \ref{thm:main3}, since the assumptions are the same).

Let $m_n$ approach infinity as $n$ approaches infinity while satisfying $m_n = o(\log \log n)$ and let $\beta \colon X \to \B$ be a finite observable. We need to show that there exist $\mb{y}_n \in \B^n$ and $\sigma_n \in \Hom(G, \Sym(n))$ such that $d_{m_n}^*(P_{\mb{y}_n}^{\sigma_n}, \beta^G_* \mu) = O(\frac{1}{\log n})$.

By Lemma \ref{lem:bowenweightapprox}, there is a sequence $\{W_n\}_{n=1}^\infty$ of weights such that $W_n$ is a denominator-$n$ $\B^{\ball{e}{m_n}}$-weight for each $n$ and $d(W_n, W_{\beta^{m_n}}) = O(\frac{\abs{\B^{\ball{e}{m_n}}}^2}{n})$. By Proposition \ref{prop:weightinverse}, for each $n$ we can pick $\mb{Y}_n,\sigma_n$ such that $W_{\sigma_n, \mb{Y}_n} = W_n$. Let $\mb{y}_n = \pi_e \mb{Y}_n$. By Proposition \ref{prop:weightballapprox},
	\[ d_{m_n}^*(P_{\mb{y}_n}^{\sigma_n}, \beta^G_* \mu) = d(W_{\sigma_n, \mb{y}_n^{m_n}}, W_{\beta^{m_n}}) = O \! \left( \! \abs{\ball{e}{m_n}} \cdot \frac{\abs{\B^{\ball{e}{m_n}}}^2}{n}\right) = O \! \left(\frac{1}{\log n} \right) . \]

\section{Counting Lemmas}
\label{sec:counting}

For a $\B$-weight $W$, let $Z_n(W)$ denote the number of pairs $(\sigma, \mb{y}) \in \Hom(G,\Sym(n)) \times \B^n$ such that $W_{\sigma,\mb{y}} = W$.
\begin{prop}
\label{prop:Zbounds}
If $W$ is a $\B$-weight with denominator $n$ then
	\[ (3\sqrt{n})^{-r \abs{\B}^2} \leq \frac{Z_n(W)}{e^{F(W) n} (n!)^r n^{(1-r)/2}} \leq (3\sqrt{n})^{r\abs{\B}^2} . \]
\end{prop}
\begin{proof}

We write
	\[ Z_n(W) = \sum_{\sigma} \abs{\{ \mb{y} \in \B^n \st W_{\sigma, \mb{y}} = W \}} = (n!)^r \EE_\sigma \abs{\{ \mb{y} \in \B^n \st W_{\sigma, \mb{y}} = W \}}. \]
where $\EE_\sigma$ denotes the expectation over a uniform choice of $\sigma \in \Hom(G,\Sym(n))$.

Proposition 2.1 of \cite{bowen2010} states that
	\[  \EE_\sigma \abs{\{ \mb{y} \in \B^n \st W_{\sigma, \mb{y}} = W \}} = \frac{n!^{1-r} \prod_{b \in \B} (n W(b))!^{2r-1}}{\prod_{i=1}^r \prod_{b,b' \in \B} (n W(b,b'; i))!} . \]
Lemma 2.2 of the same paper gives an estimate of this quantity, but for our purposes we need to be more careful about how the estimate depends on the size of the alphabet.

We use the version of Stirling's approximation
	\[ k^{k+1/2} e^{-k} \leq k! \leq 3 \cdot k^{k+1/2} e^{-k}, \]
valid for $k \geq 1$. To estimate the products that appear in the expectation, we will need to omit all factors which equal $0! = 1$ since Stirling's approximation is not valid for these. To do this carefully, let 
	\[ \B' = \{ b \in \B \st W(b) \ne 0  \} \]
and for each $i \in [r]$ let
	\[  \B'_i = \{ (b,b') \in \B^2 \st W(b,b'; i) \ne 0 \}.\]
For the numerator of the above expectation we get
\begin{align*}
	n!^{1-r} \prod_{b \in \B'} (n W(b))!^{2r-1}
	&\leq (3 n^{n+1/2}\, e^{-n})^{1-r} \prod_{b \in \B'} \left(3 (n W(b))^{n W(b) + 1/2} e^{-n W(b)} \right)^{2r-1} \\[0.2em]
	&= 3^{1-r + \abs{\B'}(2r-1)}\, n^{rn + 1/2 - r/2 + (2r-1)\abs{\B'}/2} \\
	&\qquad \times e^{-rn+(2r-1)[n \sum_{b \in \B'} W(b) \log W(b) + \frac{1}{2} \sum_{b \in \B'} \log W(b)]}
\end{align*}
and a lower bound which is identical except missing the first factor. For the denominator, let $S = \sum_{i \in [r]}\abs{\B'_i}$. We get
\begin{align*}
	\prod_{i=1}^r \prod_{(b,b') \in \B'_i} (n W(b,b'; i))!
	&\leq \prod_{i=1}^r \prod_{(b,b') \in \B'_i} 3 (n W(b,b';i))^{n W(b,b';i) + 1/2} e^{-n W(b,b';i)} \\[0.2em]
	&= 3^{S}\, n^{nr + S/2} \\
	&\qquad \times e^{n \sum_{i} \sum_{b,b'} W(b,b'; i) \log W(b,b'; i) + \frac{1}{2} \sum_{i,b,b'} \log W(b,b'; i) - nr} ,
\end{align*}
and again we have a lower bound which is identical except missing the first factor $3^S$.
Therefore the quotient is bounded above by
	\[ 3^{1-r + \abs{\B'}(2r-1)}\, n^{(1-r)/2+(2r-1)\abs{\B'}/2 - S/2}\, e^{-n F(W) + (2r-1)\frac{1}{2} \sum_b \log W(b) - \frac{1}{2}\sum_{i,b,b'} \log W(b,b'; i)}  \]
and below by
	\[ 3^{-S}\, n^{(1-r)/2+(2r-1)\abs{\B'}/2 - S/2}\, e^{-n F(W) + (2r-1)\frac{1}{2} \sum_b \log W(b) - \frac{1}{2}\sum_{i,b,b'} \log W(b,b'; i)} . \]
Since $W$ has denominator $n$, we have
	\[ 0 \geq (2r-1)\frac{1}{2} \sum_{b \in \B'} \log W(b) \geq (2r-1)\frac{1}{2} \sum_{b \in \B'} \log \frac{1}{n} = -\frac{2r-1}{2} \abs{\B'} \log n \]
and
	\[ 0 \leq  - \frac{1}{2}\sum_{i}\sum_{(b,b') \in B'_i} \log W(b,b'; i) \leq - \frac{1}{2}\sum_{i}\sum_{(b,b') \in \B'_i} \log \frac{1}{n} = \frac{S}{2} \log n . \]
Therefore $Z_n(W)$ satisfies
	\[ 3^{-S} n^{\left((1-r) - S\right)/2} e^{F(W) n} (n!)^r \leq Z_n(W) \leq 3^{1-r + \abs{\B'}(2r-1)} n^{\left((1-r)+(2r-1)\abs{\B'}\right)/2} e^{F(W) n} (n!)^r . \]
Since $S \leq r \abs{\B}^2$ and $\abs{\B'} \leq \abs{\B}$,  we conclude that
	\[ 3^{-r \abs{\B}^2} n^{\left((1-r) - r\abs{\B}^2 \right)/2} e^{F(W) n} (n!)^r \leq Z_n(W) \leq 3^{1-r + \abs{\B}(2r-1)} n^{\left((1-r)+(2r-1)\abs{\B}\right)/2} e^{F(W) n} (n!)^r , \]
and the stated inequality follows.
\end{proof}
	

The following proposition establishes the connection between the relative version of $F$ and expected numbers of good models over stochastic block models.

\begin{prop}
\label{prop:singleweightestimate}

	Given any denominator-$n$ $(\A \times \B^{\ball{e}{k}})$-weight $W_{\A\B}$, let $W_\B$ denote the $\B^{\ball{e}{k}}$-weight $\pi_{\B} W_{\A\B}$.
	Let $\mb{y} \in \B^n$ be a fixed labeling with $p_{\mb{y}} = \pi_e W_\B(\cdot)$, and let
		\[ \mu = \SBM(\mb{y}, W_\B) =  \Unif(\{ \sigma \in \Hom(G, \Sym(n)) \st W_{\sigma, \mb{y}^k} = W_\B \}) , \]
	assuming $W_\B$ is such that the desired support is nonempty.
	Then
		\[ \calE \coloneqq \EE_{\sigma \sim \mu} \abs*{\{\mb{x} \in \A^n \st W_{\sigma,(\mb{x},\mb{y}^k)} = W_{\A\B}\}} = \frac{Z_n(W_{\A\B})}{Z_n(W_\B)} . \]
	In particular,
		\[ \frac{\calE}{e^{n (F(W_{\A\B}) - F(W_\B))} } \in \left( (9n)^{-r\abs{\B}^2(\abs{\A}^2+1)},\ (9n)^{r\abs{\B}^2(\abs{\A}^2+1)} \right) . \]
\end{prop}

\begin{lemma}
	Let $W_{\A\B}$ be a $\A \times \B^{\ball{e}{k}}$ weight of denominator $n$. Then
		\[ \abs*{\{ (\sigma, \mb{x},\mb{y}) \st W_{\sigma, (\mb{x},\mb{y}^k)} = W_{\A\B}\}} \in \big\{ 0, \  \abs*{\{ (\sigma, \mb{x},\mb{Y}) \st W_{\sigma, (\mb{x},\mb{Y})} = W_{\A\B} \}} \big\}. \]
\end{lemma}
\begin{proof}
 	Suppose $\abs*{\{ (\sigma, \mb{x}, \mb{y}) \st W_{\sigma, (\mb{x}, \mb{y}^k)} = W_{\A\B}\}} \ne 0$; we then need to show
		\[ \abs*{\{ (\sigma, \mb{x}, \mb{y}) \st W_{\sigma, (\mb{x}, \mb{y}^k)} = W_{\A\B}\}} = \abs*{\{ (\sigma,\mb{x}, \mb{Y}) \st W_{\sigma, (\mb{x}, \mb{Y})} = W_{\A\B}\}} . \]
	The inequality $\leq$ is clear, since we have an injection $(\sigma, \mb{x}, \mb{y}) \mapsto (\sigma, \mb{x}, \mb{y}^k)$.
	
	The converse inequality holds because $(\sigma, \mb{x}, \mb{Y}) \mapsto (\sigma, \mb{x}, \mb{Y}_e)$ in an injection from the set on the right to the set on the left. This follows from the remark at the beginning of the proof of Proposition \ref{prop:weightballapprox}.
\end{proof}
\begin{proof}[Proof of Proposition]
Let
	\[ \tilde\mu = \Unif(\{ (\sigma, \mb{\tilde y}) \st W_{\sigma, \mb{\tilde y}^k} = W_\B\}) ; \]
then, since $\abs*{\{\mb{x} \in \A^n \st W_{\sigma,(\mb{x},\mb{\tilde y}^k)} = W_{\A\B}\}}$ is independent of the choice of $\mb{\tilde y}$ with $p_{\mb{\tilde y}} = \pi_e W_\B (\cdot)$, 
\begin{align*}
	\calE &= \EE_{(\sigma,\mb{\tilde y}) \sim \tilde\mu} \abs*{\{\mb{x} \in \A^n \st W_{\sigma,(\mb{x},\mb{\tilde y}^k)} = W_{\A\B}\}} \\
	&= \frac{\sum_{\sigma, \mb{\tilde y}} \abs*{\{\mb{x} \in \A^n \st W_{\sigma,(\mb{x},\mb{\tilde y}^k)} = W_{\A\B}\}}}{\abs*{\{ (\sigma, \mb{\tilde y}) \st W_{\sigma, \mb{\tilde y}^k} = W_\B\}}} \\
	&= \frac{\abs*{\{(\sigma, \mb{x}, \mb{\tilde y}) \st W_{\sigma,(\mb{x},\mb{\tilde y}^k)} = W_{\A\B}\}}}{\abs*{\{ (\sigma, \mb{\tilde y}) \st W_{\sigma, \mb{\tilde y}^k} = W_\B\}}} \\
	&= \frac{\abs*{\{(\sigma, \mb{x}, \mb{Y}) \st W_{\sigma,(\mb{x},\mb{Y})} = W_{\A\B}\}}}{\abs*{\{ (\sigma, \mb{Y}) \st W_{\sigma, \mb{Y}} = W_\B\}}} \tag{previous lemma}\\
	&= \frac{Z_n(W_{\A\B})}{Z_n(W_\B)} .
\end{align*}
Note that our assumption that the intended support of $\mu$ is nonempty allows us to rule out the ``0'' case in the application of the lemma.

The rest of the result then follows from our estimates on $Z_n$ in Proposition \ref{prop:Zbounds}.
\end{proof}


\section{Proof of Theorem \ref{thm:main1}}

\subsection{Upper bound}
Note that we will not rely on the Markov assumption for the upper bound.

For each $k \in \NN$,
\begin{align*}
	\inf_{\calO \ni (\alpha\beta)^G_* \mu} &\limsupinf_{n \to \infty} \frac{1}{n} \log \EE_{\sigma \sim \mu_n} \abs{\{ \mb{x} \in \A^n \st (\mb{x}, \mb{y}_n) \in \Omega(\sigma, \calO) \}} \\
	&\leq \inf_{\varepsilon} \limsupinf_{n \to \infty} \frac{1}{n} \log \EE_{\sigma \sim \mu_n} \abs{\{ \mb{x} \in \A^n \st (\mb{x}, \mb{y}_n) \in \Omega_k^*(\sigma, \alpha\tup\beta, \varepsilon) \}} \\
	&= \inf_{\varepsilon} \limsupinf_{n \to \infty} \frac{1}{n} \log \EE_{\sigma \sim \mu_n} \abs{\{ \mb{x} \in \A^n \st (\mb{x}^k, \mb{y}_n^k) \in \Omega_0^*(\sigma, (\alpha\tup\beta)^k, \varepsilon) \}} \\
	&\leq \inf_{\varepsilon} \limsupinf_{n \to \infty} \frac{1}{n} \log \EE_{\sigma \sim \mu_n} \abs{\{ \mb{X} \in (\A^{\ball{e}{k}})^n \st (\mb{X}, \mb{y}_n^k) \in \Omega_0^*(\sigma, (\alpha\tup\beta)^k, \varepsilon) \}} .
\end{align*}
Write
\begin{align*}
	\calE_k(n,\varepsilon) &\coloneqq \EE_{\sigma \sim \mu_n} \abs{\{ \mb{X} \in (\A^{\ball{e}{k}})^n \st (\mb{X}, \mb{y}_n^k) \in \Omega_0^*(\sigma, (\alpha\tup\beta)^k, \varepsilon) \}} \\
	&= \EE_{\sigma \sim \mu_n} \abs{\{ \mb{X} \in (\A^{\ball{e}{k}})^n \st d\big(W_{\sigma,(\mb{X}, \mb{y}_n^k)} , W_{(\alpha\tup\beta)^k} \big) < \varepsilon) \}}
\end{align*}
and assume that $n$ is large enough that $m_n \geq k$.

Writing $\calW_n(\alpha\beta,k, \varepsilon)$ for the set of all denominator-$n$ weights $W$ with $d(W, W_{(\alpha\beta)^k}) < \varepsilon$,
\begin{align*}
	\calE_k(n, \varepsilon)
		&= \EE_{\sigma \sim \mu_n} \!\sum_{W \in \calW_n(\alpha\tup\beta, k, \varepsilon)} \! \abs{\{ \mb{X} \in (\A^{\ball{e}{k}})^n \st W_{\sigma, (\mb{X}, \mb{y}_n^{k})} = W \}} \\
		&= \sum_{W \in \calW_n(\alpha\tup\beta, k, \varepsilon)} \EE_{\sigma \sim \mu_n} \big[ \abs{\{ \mb{X} \in (\A^{\ball{e}{k}})^n \st W_{\sigma, (\mb{X}, \mb{y}_n^{k})} = W \}} \big| W_{\sigma, \mb{y}_n^k} = \pi_{\B} W \big] \PP_{\sigma \sim \mu_n}(W_{\sigma, \mb{y}_n^k} = \pi_{\B} W)
\end{align*}
since if $W_{\sigma, \mb{y}_n^k} \ne \pi_{\B} W$ then $W_{\sigma, (\mb{X}, \mb{y}_n^{k})} \ne W$. But $\mu_n$ conditioned on $\{W_{\sigma, \mb{y}_n^k} = \pi_{\B} W\}$ is $\SBM(\mb{y}_n, \pi_\B W)$, so we can bound the expectation above using Proposition \ref{prop:singleweightestimate}, getting
	\[ \calE_k(n, \varepsilon) \leq (9n)^{r\abs{\B^{\ball{e}{k}}}^2(\abs{\A^{\ball{e}{k}}}+1)} \!\sum_{W \in \calW_n(\alpha\tup\beta, k, \varepsilon)} \! e^{n ( F(W) - F(\pi_{\B} W))} \PP_{\sigma \sim \mu_n}(W_{\sigma, \mb{y}_n^k} = \pi_{\B} W) . \]

Note $(9n)^{r\abs{\B^{\ball{e}{k}}}^2(\abs{\A^{\ball{e}{k}}}+1)} \leq e^{o_{n \to \infty}(n)}$. Fix $\delta > 0$. By continuity of $F$, for all small enough $\varepsilon$ (possibly depending on $k$) we have
	\[ \calE_k(n,\varepsilon) \leq e^{n ( F_\mu(T, \alpha^k \mid \beta^k) + \delta + o_{n \to \infty}(1))} \sum_{W \in \calW_n(\alpha\tup\beta, k, \varepsilon)} \PP_{\sigma \sim \mu_n}(W_{\sigma, \mb{y}_n^k} = \pi_{\B} W) . \]

Bounding each probability by 1, we get
	\[ \calE_k(n,\varepsilon) \leq e^{n ( F_\mu(T, \alpha^k \mid \beta^k) + \delta + o_{n \to \infty}(1))} \abs*{\calW_n(\alpha\tup\beta, k, \varepsilon)} . \]
But
	\[ \abs*{\calW_n(\alpha\tup\beta, k, \varepsilon)} \leq n^{r\abs*{(\A \times \B)^{\ball{e}{k}}}^2} \leq e^{o_{n \to \infty}(n)} , \]
so this implies
\begin{align*}
	\limsupinf_{n \to \infty} \frac{1}{n} \log \calE_k(n, \varepsilon)
	&\leq F_\mu (T, \alpha^k \mid \beta^k) + \delta \\
	&\leq F_\mu (T, \alpha^k \mid \beta^{k_2}) + \delta
\end{align*}
for any $k_2 \geq k$, by monotonicity under splitting.
Taking the limit as $k_2 \to \infty$ followed by the infimum over $\varepsilon$ (which takes $\delta$ to 0) and $k$ gives
	\[ \inf_{\varepsilon,k} \limsupinf_{n \to \infty} \frac{1}{n} \log \calE_k(n, \varepsilon) \leq f_\mu (T, \alpha \mid \beta). \]
Since
	\[ \inf_{\calO \ni (\alpha\beta)^G_* \mu} \limsupinf_{n \to \infty} \frac{1}{n} \log \EE_{\sigma \sim \mu_n} \abs{\{ \mb{x} \in \A^n \st (\mb{x}, \mb{y}_n) \in \Omega(\sigma, \calO) \}} \leq \inf_{\varepsilon} \limsupinf_{n \to \infty} \frac{1}{n} \log \calE_k(n, \varepsilon)  \]
for every $k$, this completes the upper bound.

\subsection{Lower bound}

Fix $k \in \NN$. To estimate
\begin{align*}
	\calE &\coloneqq \EE_{\sigma \sim \mu_n} \abs*{\{ \mb{x} \in \A^n \st (\mb{x},\mb{y}_n) \in \Omega_k^*(\sigma, \alpha \tup \beta, \varepsilon) \}}
\end{align*}
we bound below using the expected size of
	\[ \mathcal{X}_k (\sigma, \alpha\beta, \varepsilon \mid \mb{y}_n) \coloneqq \{ \mb{X} \in \left( \A^{\ball{e}{k}}\right)^n \st (\mb{X}, \mb{y}_n^k) \in \Omega_0^*(\sigma, (\alpha\beta)^k, \varepsilon) \} . \]
This is not a true lower bound but, by Equation \ref{eqn:Qbound} below, there are constants $C,d,c$ independent of $n$ such that
	\[ \abs*{\mathcal{X}_k (\sigma, \alpha\beta, \varepsilon \mid \mb{y}_n)} \leq C \exp \big(n d \varepsilon + n \shent(2 \abs{\ball{e}{k}} \varepsilon) \big) \cdot \abs*{\{ \mb{x} \in \A^n \st (\mb{x},\mb{y}_n) \in \Omega_k^*(\sigma, \alpha \tup \beta, \varepsilon) \}} . \]
The `error' factor has an exponential growth rate which vanishes as $\varepsilon \to 0$, so will not be a problem.

We now find a lower bound for the expectation of $\abs{\mathcal{X}_k}$. Applying Proposition \ref{prop:singleweightestimate} as above, we have
\begin{align*}
	\EE_{\sigma \sim \mu_n} &\abs{\mathcal{X}_{k} (\sigma, \alpha\tup\beta, \varepsilon \mid \mb{y}_n)} \\
	&= \sum_{W \in \calW_n(\alpha\tup\beta, k, \varepsilon)} \! \EE_{\sigma \sim \mu_n} \abs{\{ \mb{X} \in (\A^{\ball{e}{k}})^n \st W_{\sigma, (\mb{X}, \mb{y}_n^{k})} = W \}} \\
	&\geq \sum_{W \in \calW_n(\alpha\tup\beta, k, \varepsilon)} \! \exp \!\big[ n( F(W) - F(\pi_\B W) - o_n(1)) \big] \PP_{\sigma \sim \mu_n}\big(\pi_\B W = W_{\sigma, \mb{y}_n^k}\big) .
\end{align*}

For any $\delta>0$, for small enough $\varepsilon>0$ (independent of $n$), by continuity of $F$ this is at least
	\[ \exp \big[ n( F_\mu(\alpha^k \mid \beta^k) - \delta - o_n(1)) \big] \sum_{W \in \calW_n(\alpha\tup\beta, k, \varepsilon)} \PP_{\sigma \sim \mu_n}\big(\pi_\B W = W_{\sigma, \mb{y}_n^k}\big) . \]

We give a lower bound for the sum by first rewriting it as
	\[ \sum_{\mathclap{W_\B \text{ denom.-} n\ \B^{\ball{e}{k}}-\text{weight}}}\ \abs*{\{ W \in \calW_n(\alpha\tup\beta, k, \varepsilon) \st \pi_\B W = W_\B \}} \cdot \PP_{\sigma \sim \mu_n}(W_{\sigma, \mb{y}_n^k} = W_\B) . \]

Fix $\eta > 0$. By Lemma \ref{lem:denomnapprox}, for all large enough $n$ the $\B$-weight $W_{\sigma_n, \mb{y}_n}$ can be extended to a $\B^{\ball{e}{k}}$-weight $W_{\B}$ with $d(W_\B, W_{\beta^k}) \leq \eta$; to apply the lemma we can think of the extended weight $W_{\B}$ as having alphabet $\B^{\ball{e}{k} \setminus \{e\}} \times \B$, and recall that we assume $\lim_{n \to \infty} d(W_{\sigma_n, \mb{y}_n}, W_\beta) = 0$. Choose $\sigma, \mb{Y}$ such that $W_{\B} = W_{\sigma, \mb{Y}}$. Since $W_{\B}$ is an extension of $W_{\sigma_n, \mb{y}_n}$, we can make this choice in such a way that $\pi_e \mb{Y} = \mb{y}_n$.

Let $\widetilde{W_\B} = W_{\sigma, \mb{y}_n^k}$. By Proposition \ref{prop:weightballapprox},
	\[ d(\widetilde{W_\B}, W_{\beta^k}) \leq d(\widetilde{W_\B}, W_\B) + d(W_\B, W_{\beta^k}) \leq  2r \abs{\ball{e}{k}} \eta + \eta . \]
So, as long as $\eta$ is small enough and $n$ is large enough (depending on $\varepsilon,k$), by Lemma \ref{lem:denomnapprox}
	\[ \abs*{\{ W \in \calW_n(\alpha\tup\beta, k, \varepsilon) \st \pi_\B W = W_\B \}} \geq 1 . \]
Now consider the probability appearing in the $\widetilde{W_\B}$ term:
	\[ \PP_{\sigma \sim \mu_n}(W_{\sigma, \mb{y}_n^k} = \widetilde W_\B)
		= \frac{\abs{\{ \sigma \st W_{\sigma, \mb{y}_n^k} = \widetilde W_\B\}}}{ \abs{\{ \sigma \st W_{\sigma, \mb{y}_n} = W_{\sigma_n, \mb{y}_n}\}}} . \]
By symmetry in choice of $\mb{y}$ with the correct letter frequencies, we can write this as
\begin{align*}
	\PP_{\sigma \sim \mu_n}(W_{\sigma, \mb{y}_n^k} = \widetilde W_\B)
		&= \frac{\abs*{\{ (\sigma,\mb{y}) \st W_{\sigma, \mb{y}^k} = \widetilde W_\B\}}}{\abs*{\{(\sigma,\mb{y}) \st W_{\sigma, \mb{y}} = W_{\sigma_n, \mb{y}_n}\}}} \\
		&= \frac{\abs*{\{ (\sigma,\mb{Y}) \st W_{\sigma, \mb{Y}} = \widetilde W_\B\}}}{\abs*{\{(\sigma,\mb{y}) \st W_{\sigma, \mb{y}} = W_{\sigma_n, \mb{y}_n}\}}} \tag{Prop. \ref{prop:weightballapprox}}\\
		&= \frac{Z_n(\widetilde W_\B)}{Z_n (W_{\sigma_n, \mb{y}_n})} \tag{definition of $Z_n$}\\
		&\geq \exp \!\left( n[ F(\widetilde W_\B) - F(W_{\sigma_n, \mb{y}_n})] \right) \cdot (3\sqrt{n})^{-r(\abs{\B^{\ball{e}{k}}}^2 - \abs{\B})} \tag{Prop. \ref{prop:Zbounds}}\\
		&= \exp \!\left( n[ F(\widetilde W_\B) - F(W_{\sigma_n, \mb{y}_n}) - o(1)] \right) . \\
\end{align*}
By continuity of $F$, we then get
	\[ \PP_{\sigma \sim \mu_n}(W_{\sigma, \mb{y}_n^k} = \widetilde W_\B) 
		\geq \exp n \!\left(F_\mu(\beta^k)- F_\mu(\beta) - 2\delta + o(1) \right) \]
for all large enough $n$ and small enough $\eta$ (again depending on $k,\varepsilon$), with $\delta > 0$ the same as chosen above.
Since $\beta^G_* \mu$ is a Markov chain, $F_\mu(\beta^k) = F_\mu(\beta)$.

Putting this all together: for any $k \in \NN$, for all $\delta >0$ we have
	\[ \EE_{\sigma \sim \mu_n} \abs{\mathcal{X}_{k} (\sigma, \alpha\tup\beta, \varepsilon \mid \mb{y}_n)} \geq \exp \big[ n( F_\mu(\alpha^k \mid \beta^k) - 3 \delta - o(1)) \big]  \]
for all large enough $n$ and small enough $\varepsilon>0$.
	
It follows that for any $k \in \NN$
	\[ \inf_\varepsilon \limsup_{n \to \infty} \frac{1}{n} \log \EE_{\sigma \sim \mu_n} \abs*{\{ \mb{x} \in \A^n \st (\mb{x},\mb{y}_n) \in \Omega_k^*(\sigma, \alpha \tup \beta, \varepsilon) \}} \geq F_\mu (T, \alpha^k \mid \beta^k) . \]
Taking the limit as $k \to \infty$ gives the desired bound, using Corollary \ref{cor:markovF} and that the family of pseudometrics $\{d^*_k \st k \in \NN\}$ generates the weak$^*$ topology.


\section{Proof of Theorem \ref{thm:main2}}

Let $W_n = W_{\sigma_n, \mb{y}_n^{m_n}}$, so that
	\[ \mu_n = \SBM(\mb{y}_n, W_n) . \]
Note that, by definition of $\mu_n$, 
	\[ \PP_{\sigma \sim \mu_n} \left( W_{\sigma, \mb{y}_n^{m_n}} = W_n \right) = 1 . \]

\begin{lemma}
\label{lem:Flimit}
	With $W_n$ as just defined in terms of $m_n$, $\sigma_n$, and $\mb{y}_n$, we have
		\[ \lim_{n \to \infty} F(W_n) = f_\mu(T, \beta) . \]
\end{lemma}

\begin{proof}
	The assumption in the theorem statement that $d_{m_n}^*(P_{\mb{y}_n}^{\sigma_n}, \beta^G_*\mu) = O \big( \tfrac{1}{\log n} \big)$ implies the existence of a constant $C$ such that
		\[ d(W_n, W_{\beta^{m_n}}) \leq \frac{C}{\log n} . \]
	By Lemma \ref{lem:Fcontinuity} we have
	\[
		\abs{F(W_{\sigma, \mb{y}^{m_n}}) - F(W_{\beta^{m_n}})}
		\leq 4r \big( \shent(\tfrac{C}{\log n}) + \tfrac{C}{\log n} \abs{\ball{e}{m_n}} \log_2 \abs{\B} \big)
		= o(1)
	\]
	using that $m_n = o(\log\log n)$.
	Since $m_n$ approaches infinity as $n$ goes to infinity we have $f_\mu (T, \beta) = \lim_{n \to \infty} F(W_{\beta^{m_n}})$, so the result follows.
\end{proof}

\begin{lemma}
\label{lem:expbound}
	If $m_n = o(\log \log n)$, then for any $k > 0$ and $\varepsilon > 0$ we have $\abs{\B^{\ball{e}{m_n}}}^k = o(n^{\varepsilon})$.
\end{lemma}
\begin{proof}
	This is certainly true if $\abs{\B} = 1$; assume therefore that $\abs{\B} \geq 2$.
	
	Our assumption $m_n = o(\log\log n)$ guarantees that
		\[ (2r-1)^{m_n} < \frac{r-1}{r} \frac{\varepsilon}{k \log \abs{\B}}\log n \]
	for all large enough $n$. Therefore
		\[ \abs{\ball{e}{m_n}} =  \frac{r(2r-1)^{m_n} - 1}{r-1} < \frac{\varepsilon}{k \log \abs{\B}} \log n . \]
	This inequality can be rearranged to give
		\[ \abs{\B^{\ball{e}{m_n}}}^k < n^{\varepsilon} . \]
	Since $\varepsilon>0$ is arbitrary, the result follows.
\end{proof}

\bigbreak

In the remainder of this section we prove Theorem \ref{thm:main2} by first proving the right-hand side is an upper bound for the left, then proving it is also lower bound.

\subsection{Upper bound}
Just as in the proof of the upper bound in Theorem \ref{thm:main1},
for each $k \in \NN$ and $\varepsilon>0$ we have
	\[ \inf_{\calO \ni (\alpha\beta)^G_* \mu} \limsupinf_{n \to \infty} \frac{1}{n} \log \EE_{\sigma \sim \mu_n} \abs{\{ \mb{x} \in \A^n \st (\mb{x}, \mb{y}_n) \in \Omega(\sigma, \calO) \}} \leq \limsup_{n \to \infty} \frac{1}{n} \log \calE_k (n, \varepsilon), \]
where
\begin{align*}
	\calE_k(n,\varepsilon) &\coloneqq \EE_{\sigma \sim \mu_n} \abs{\{ \mb{X} \in (\A^{\ball{e}{k}})^n \st (\mb{X}, \mb{y}_n^k) \in \Omega_0^*(\sigma, (\alpha\tup\beta)^k, \varepsilon) \}} \\
	&= \EE_{\sigma \sim \mu_n} \abs{\{ \mb{X} \in (\A^{\ball{e}{k}})^n \st d\big(W_{\sigma,(\mb{X}, \mb{y}_n^k)} , W_{(\alpha\tup\beta)^k} \big) < \varepsilon) \}} .
\end{align*}
We assume that $n$ is large enough that $m_n \geq k$.

Since $\mu_n$ is $\SBM(\sigma_n, \mb{y}_n, m_n)$ rather than $\SBM(\sigma_n, \mb{y}_n, k)$, we cannot apply Proposition \ref{prop:singleweightestimate} directly to this expression.  We get around this as follows: Let
	\[ \calW_n(m, m') \coloneqq \left\{ W_{\sigma, (\mb{X}, \mb{y}^{m'})} \st \sigma \in \Hom(G,\Sym(n)),\ \mb{X} \in (\A^{\ball{e}{m}})^n,\ \mb{y} \in \B^n \right\} . \]
All elements of this set are denominator-$n$ $\A^{\ball{e}{m}} \times \B^{\ball{e}{m'}}$-weights; we avoid the question of exactly which weights are in this set, but call such weights \emph{attainable}. For $k \leq m$ and $k' \leq m'$ let
	\[
		\calW_n(m, m'; \alpha\tup\beta, k, k'; \varepsilon)
		= \left\{ W \in \calW_n(m, m') \st d\!\left(\pi_{k,k'} W,\ W_{\alpha^{k} \tup \beta^{k'}}\right) < \varepsilon \right\}
	\]
denote the set of such weights whose appropriate marginal is within $\varepsilon$ of the $(\A^{\ball{e}{k}} \times \B^{\ball{e}{k'}})$-weight $W_{\alpha^{k} \tup \beta^{k'}}$. For now we take $m=k=k'$ but we will need more generality below.
Then
	\[ \calE_k(n, \varepsilon)
		= \EE_{\sigma \sim \mu_n} \!\sum_{W \in \calW_n(k, m_n; \alpha\tup\beta, k,k; \varepsilon)} \! \abs{\{ \mb{X} \in (\A^{\ball{e}{k}})^n \st W_{\sigma, (\mb{X}, \mb{y}_n^{m_n})} = W \}} \]
so we can apply Proposition \ref{prop:singleweightestimate} to get
\begin{align*}
	\calE_k(n,\varepsilon)
	&\leq (9n)^{r\abs{\B^{\ball{e}{m_n}}}^2(\abs{\A^{\ball{e}{k}}}+1)} \!\sum_{W \in \calW_n(k, m_n; \alpha\tup\beta, k,k; \varepsilon)} \! e^{n ( F(W) - F(\pi_{\B} W))} \1{\pi_{\B} W = W_n} .
\end{align*}

By Lemma \ref{lem:expbound} we have $(9n)^{r\abs{\B^{\ball{e}{m_n}}}^2(\abs{\A^{\ball{e}{k}}}+1)} \leq e^{o_{n \to \infty}(n)}$. Using this and Lemma \ref{lem:Flimit} we have
	\[ \calE_k(n,\varepsilon) \leq \sum_{W \in \calW_n(k, m_n; \alpha\tup\beta, k,k; \varepsilon)} \! e^{n ( F(W) - f(T, \beta) + o_{n \to \infty}(1))}  \1{\pi_{\B} W = W_n} , \]
where the little $o$ is uniform over all terms in the sum. Here we use the assumption that $f_\mu(T, \beta)$ is finite.

By definition of $\calW_n(k, m_n)$, for any $W \in \calW_n(k, m_n; \alpha\tup\beta, k,k; \varepsilon)$ we can pick $\sigma \in \Hom(G,\Sym(n))$, $\mb{X} \in (\A^{\ball{e}{k}})^n$, and $\mb{y} \in \B^n$ so that $W = W_{\sigma, (\mb{X}, \mb{y}^{m_n})}$. Then since $\mb{X} \tup \mb{y}^{m_n}$ is a splitting of $\mb{X} \tup \mb{y}^k$, by Lemma \ref{lem:splittings} we have
	\[ F(W) = F(\sigma, \mb{X} \tup \mb{y}^{m_n}) \leq F(\sigma, \mb{X} \tup \mb{y}^k) = F(\pi_{k,k} W) . \]
By continuity of $F$, for all small enough $\varepsilon$ (depending on $k$) we have
	\[ F(\pi_{k,k} W) \leq F(W_{(\alpha\tup\beta)^k}) + \delta = F_\mu (T, (\alpha\tup\beta)^k) + \delta. \]
Along with the above, this implies that
	\[ \calE_k(n,\varepsilon) \leq e^{n ( F(T,(\alpha\tup\beta)^k) - f(T, \beta) + o_n(1)+ \delta)}  \! \sum_{W \in \calW_n(k, m_n; \alpha\tup\beta, k,k; \varepsilon)}  \1{\pi_{\B} W = W_n} . \]
Bounding all terms in the sum by 1, we get
	\[ \calE_k(n, \varepsilon) \leq e^{n(F(T,(\alpha\tup\beta)^k) - f_\mu (T, \beta) + o_n(1) + \delta)}\, \abs*{\calW_n(k, m_n; \alpha\tup\beta, k,k; \varepsilon)} . \]
Using Lemma \ref{lem:expbound} we have
	\[ \abs*{\calW_n(k, m_n; \alpha\tup\beta, k,k; \varepsilon)} \leq \abs*{\calW_n(k, m_n)} \leq n^{r\abs*{\A^{\ball{e}{k}} \times \B^{\ball{e}{m_n}}}^2} \leq e^{o_{n \to \infty}(n)} , \]
so this implies
	\[ \limsupinf_{n \to \infty} \frac{1}{n} \log \calE_k(n, \varepsilon) \leq F_\mu (T, (\alpha\tup\beta)^k) -f_\mu (T,\beta)+ \delta . \]
Taking the infimum over $\varepsilon$ and $k$, and using the chain rule for $f$ (Corollary \ref{cor:chainrule}, again using the assumption that $f_\mu(T, \beta)$ is finite), gives
	\[ \inf_{\varepsilon,k} \limsupinf_{n \to \infty} \frac{1}{n} \log \calE_k(n, \varepsilon) \leq f_\mu (T, \alpha\tup\beta) - f_\mu (T, \beta) = f_\mu (T, \alpha \mid \beta). \]
Since
	\[ \inf_{\calO \ni (\alpha\beta)^G_* \mu} \limsupinf_{n \to \infty} \frac{1}{n} \log \EE_{\sigma \sim \mu_n} \abs{\{ \mb{x} \in \A^n \st (\mb{x}, \mb{y}_n) \in \Omega(\sigma, \calO) \}} \leq \inf_{\varepsilon} \limsupinf_{n \to \infty} \frac{1}{n} \log \calE_k(n, \varepsilon) , \]
for every $k$, this completes the upper bound.

\subsection{Lower bound}
In this section we denote
\begin{align*}
	\mathcal{X}_{k_1, k_2} (\sigma, \alpha\tup\beta, \varepsilon \mid \mb{y})
	\coloneqq{}& \{ \mb{X} \in \big(\A^{\ball{e}{k_1}}\big)^n \st (\mb{X}, \mb{y}^{k_2}) \in \Omega_0^*(\sigma, \alpha^{k_1} \tup \beta^{k_2}, \varepsilon) \} \\[0.2cm]
	\Omega_k^*(\sigma, \alpha\tup\beta, \varepsilon \mid \mb{y})
	\coloneqq{}& \{ \mb{x} \in \A^n \st (\mb{x}, \mb{y}) \in \Omega_k^*(\sigma, \alpha\tup\beta, \varepsilon) \} \\
\intertext{(note the dependence on $n$ is implicitly specified by $\sigma \in \Hom(G, \Sym(n))$ and $\mb{y} \in \B^n$), and with $\Sigma = \{\mu_n\}_{n=1}^\infty$}
	\h_{\Sigma}(\mu, \alpha \mid \beta \st k, \varepsilon)
	\coloneqq{}& \limsup_{n \to \infty} \frac{1}{n} \log \EE_{\sigma \sim \mu_n} \abs*{ \{ \mb{x} \in \A^n \st (\mb{x}, \mb{y}) \in \Omega_k^*(\sigma, \alpha\tup\beta, \varepsilon) \} } \\
	={}& \limsup_{n \to \infty} \frac{1}{n} \log \EE_{\sigma \sim \mu_n} \abs*{ \Omega_k^*(\sigma, \alpha\tup\beta, \varepsilon \mid \mb{y}) } .
\end{align*}

The following two claims are used to relate the sizes of the sets defined above.

\begin{claim}
Let $k \leq \min(k_1,k_2)$. For any $\sigma,\mb{y}$ we have
	\[ \pi_e \left[ \mathcal{X}_{k_1, k_2} (\sigma, \alpha\tup\beta, \varepsilon \mid \mb{y}) \right] \subseteq  \Omega_k^*(\sigma, \alpha\tup\beta, c\varepsilon \mid \mb{y})  \]
where $c = 1+\abs{\ball{e}{k}}$.
\end{claim}
\begin{proof}
	If $(\mb{X}, \mb{y}^{k_2}) \in \Omega_0^*(\sigma, \alpha^{k_1} \tup \beta^{k_2}, \varepsilon)$, then
		\[
			\pi_{k,k} (\mb{X},\mb{y}^{k_2})
				\in \Omega_0^*(\sigma, (\alpha\tup\beta)^k, \varepsilon) ;
		\]
	this follows from the fact that total variation distance is nonincreasing under pushforwards.
	Applying Proposition \ref{prop:weightballapprox}, we get
		\[
			(\pi_e \mb{X}, \mb{y}) = \pi_e \left(\pi_{k,k} (\mb{X},\mb{y}^{k_2}) \right)
				\in \Omega_k^*(\sigma, \alpha\tup\beta, c\varepsilon) . \qedhere
		\]
\end{proof}

\begin{claim}
	Fix $\sigma, \mb{y}$, and $k \leq \min(k_1, k_2)$. As established in the previous claim, we can consider $\pi_e$ as a map from $ \mathcal{X}_{k_1, k_2} (\sigma, \alpha\tup\beta, \varepsilon \mid \mb{y})$ to $\Omega_k^*(\sigma, \alpha\tup\beta, c\varepsilon \mid \mb{y})$.
	There are constants $C,d$ independent of $n$ such that $\pi_e$ is at most $C \exp \big(nd\varepsilon + n \shent(2 \abs{\ball{e}{k}} \varepsilon) \big)$-to-one.
\end{claim}
\begin{proof}
If $\Omega_k^*(\sigma, \alpha\tup\beta, c\varepsilon \mid \mb{y})$ is empty, then the claim is vacuously true.
Otherwise, fix $\mb{x} \in \Omega_k^*(\sigma, \alpha\tup\beta, c\varepsilon \mid \mb{y})$. If $\mb{X} \in \pi_e^{-1}\{\mb{x}\}$, then $\pi_e (\mb{X}, \mb{y}^{k}) = (\mb{x}, \mb{y})$. By Claim 3 in the proof of Proposition 3.2 of \cite{bowen2010} the number of such pairs $(\mb{X}, \mb{y}^k)$, and therefore the number of such $\mb{X}$, is bounded above by
	\[3\sqrt{2} \abs{\A \times \B}^{\abs{\ball{e}{k}} \big( n\abs{\ball{e}{k}}\varepsilon - 1 \big) } \exp \big( n \shent(2 \abs{\ball{e}{k}} \varepsilon) \big) \]
where $\shent$ is the Shannon entropy. (We give more explicit constants here than in \cite{bowen2010} to make the dependence on $n$ clear).
\end{proof}

Claim 2 implies that
\begin{equation}
\label{eqn:Qbound}
	\abs*{ \mathcal{X}_{k_1, k_2} (\sigma, \alpha\tup\beta, \varepsilon \mid \mb{y}) } \leq C \exp \big(n d \varepsilon + n \shent(2 \abs{\ball{e}{k}} \varepsilon) \big) \cdot \abs*{ \Omega_k^*(\sigma, \alpha\tup\beta, c\varepsilon \mid \mb{y}) } ,
\end{equation}
where $C,d$ are independent of $n$. \\

We now find a lower bound for the expectation of $\abs{\mathcal{X}}$. Fix $k_1, k_2 \in \NN$, and suppose $n$ is large enough that $m_n \geq \max(k_1, k_2)$. Using Proposition \ref{prop:singleweightestimate} and Lemma \ref{lem:expbound}, we have
\begin{align*}
	\EE_{\sigma \sim \mu_n} &\abs{\mathcal{X}_{k_1, k_2} (\sigma, \alpha\tup\beta, \varepsilon \mid \mb{y}_n)} \\
	&= \sum_{W \in \calW_n(k_1, m_n; \alpha\tup\beta, k_1, k_2; \varepsilon)} \! \EE_{\sigma \sim \mu_n} \abs{\{ \mb{X} \in (\A^{\ball{e}{k_1}})^n \st W_{\sigma, (\mb{X}, \mb{y}_n^{m_n})} = W \}} \\
	&\geq \sum_{W \in \calW_n(k_1, m_n; \alpha\tup\beta, k_1, k_2; \varepsilon)} \! \exp \!\big[ n( F(W) - F(\pi_\B W) - o_n(1)) \big] \1{\pi_\B W = W_{\sigma, \mb{y}_n^{m_n}}} \\
	&\geq \inf_{W \in \calW_n(k_1, m_n; \alpha\tup\beta, k_1, k_2; \varepsilon)} \exp \!\big[ n( F(W) - F(\pi_\B W) - o_n(1)) \big] \\
	&\quad \times \sum_{W \in \calW_n(k_1, m_n; \alpha\tup\beta, k_1, k_2; \varepsilon)} \1{\pi_\B W = W_{\sigma, \mb{y}_n^{m_n}}}
\end{align*}
We bound the infimum below as follows: Given any $W \in \calW_n(k_1, m_n; \alpha\tup\beta, k_1, k_2; \varepsilon)$, we can let $\mb{X}, \mb{y}, \sigma$ be such that $W = W_{\sigma, (\mb{X}, \mb{y}^{m_n})}$. Then by Lemma \ref{lem:splittings} and continuity of $F$
\begin{align*}
	F(W) - F(\pi_\B W)
	&= F(\sigma, \mb{X} | \mb{y}^{m_n}) \\
	&\geq F(\sigma, \mb{X} | \mb{y}^{k_2}) \\
	&= F(\pi_{k_1, k_2} W) - F(\pi_\B \pi_{k_1, k_2} W) \\
	&\geq F(T, \alpha^{k_1} | \beta^{k_2}) - \delta
\end{align*}
for any $\delta>0$ for all small enough $\varepsilon$ (with ``small enough'' dependent only on $k_1, k_2$). This implies that the infimum is bounded below by
	\[ \exp \!\big[ n(F(T, \alpha^{k_1} | \beta^{k_2}) - o_n(1) - \delta) \big] . \]

We bound the sum below by first rewriting it as
	\[ \abs*{\{ W \in \calW_n(k_1, m_n; \alpha\tup\beta, k_1, k_2; \varepsilon) \st \pi_\B W = W_{\sigma, \mb{y}_n^{m_n}} \}} . \]
The following claim, then, implies that the sum is bounded below by 1.

\begin{claim}
	For all large enough $n$,
		\[ \big\{ W \in \calW_n(k_1, m_n; \alpha\tup\beta, k_1, k_2; \varepsilon) \st \pi_\B W = W_{\sigma, \mb{y}_n^{m_n}} \big\} \neq \varnothing . \]
\end{claim}
\begin{proof}
By Lemma \ref{lem:denomnapprox}, if
	\[ n > 680 \abs{\A^{\ball{e}{k_1}} \times \B^{\ball{e}{m_n}}}^2 r / \varepsilon \]
and $d(W_{\sigma, \mb{y}_n^{m_n}}, W_{\beta^{m_n}}) < \frac{\varepsilon}{530 r}$ then there is a $(\A^{\ball{e}{k_1}} \times \B^{\ball{e}{m_n}})$-weight $W$ with $\pi_\B W = W_{\sigma, \mb{y}_n^{m_n}}$ and $d(W, W_{\alpha^{k_1} \tup \beta^{m_n}}) < \varepsilon$.
By definition of $\mu_n$ and Lemma \ref{lem:expbound}, both conditions are met for all large enough $n$.

The claim will follow if we show that $W$ is attainable.

With $W$ as chosen above, by Proposition \ref{prop:weightinverse} we can choose $\tilde\sigma \in \Hom(G, \Sym(n))$, $\mb{\tilde X} \in (\A^{\ball{e}{k_1}})^n$, and $\mb{\tilde Y} \in (\B^{\ball{e}{m_n}})^n$ such that $W = W_{\tilde\sigma, (\mb{\tilde X}, \mb{\tilde Y})}$.

Let $\mb{\tilde y} = \pi_e \mb{\tilde Y} \in \B^n$. To complete the proof we show that $\mb{\tilde y}^{m_n} = \mb{\tilde Y}$, i.e.
	\[ \mb{\tilde y}\big(\tilde \sigma(g) i \big) = \left( \mb{\tilde Y}(i) \right)_g \]
for all $i \in [n]$ and $g \in \ball{e}{m_n}$. We prove this by induction on the word length $\abs{g}$.

The base case $\abs{g} = 0$ (i.e. $g=e$) follows immediately from the definition of $\mb{\tilde y}$.

For the inductive step, write $g = ht$ with $\abs{h} = \abs{g}-1$ and $t \in \{s_1^{\pm 1}, \ldots, s_r^{\pm 1}\}$. Then, assuming the result holds for $h$,
	\[ \mb{\tilde y}\big(\tilde \sigma(g) i \big) = \mb{\tilde y}\big(\tilde\sigma(h) \tilde\sigma(t) i \big) = \left( \mb{\tilde Y}(\tilde\sigma(t)i) \right)_h. \]
Now since $W_{\tilde\sigma, \mb{\tilde Y}} = W_{\sigma_n, \mb{y}_n^{m_n}}$, we can pick $j \in [n]$ such that
	\[ \mb{\tilde Y}(i) = \mb{y}_n^{m_n}(j) \quad \text{and} \quad \mb{\tilde Y}(\tilde\sigma(t) i) = \mb{y}_n^{m_n}(\sigma(t) j) . \]
This implies
	\[ \left( \mb{\tilde Y}(\tilde\sigma(t)i) \right)_h = \big( \mb{y}_n^{m_n}(\sigma(t) j) \big)_h = \mb{y}_n(\sigma(g) j) = \big( \mb{y}_n^{m_n}(j) \big)_g = \left( \mb{\tilde Y}(i) \right)_g. \qedhere \]
\end{proof}

Hence for all large enough $n$ we have
	\[ \EE_{\sigma \sim \mu_n} \abs{\mathcal{X}_{k_1, k_2} (\sigma, \alpha\tup\beta, \varepsilon \mid \mb{y}_n)} \geq \exp \big[ n ( F(T, \alpha^{k_1} \mid \beta^{k_2}) - o_n(1) -  \delta) \big] , \]
and therefore
	\[ \limsup_{n \to \infty} \frac{1}{n} \log \EE_{\sigma \sim \mu_n} \abs{\mathcal{X}_{k_1, k_2} (\sigma, \alpha\tup\beta, \varepsilon \mid \mb{y}_n)} \geq F(T, \alpha^{k_1} \mid \beta^{k_2}) - \delta . \]
	
Combining this lower bound with Equation (\ref{eqn:Qbound}) and the definition of $\h_\Sigma (\mu, \alpha \mid \beta \st k, c \varepsilon)$, we get
	\[  d \varepsilon + \shent(2 \abs{\ball{e}{k}} \varepsilon) + \h_\Sigma (\mu, \alpha \mid \beta \st k, c \varepsilon) \geq F(T, \alpha^{k_1} \mid \beta^{k_2}) - \delta . \]
Taking the inf in $\varepsilon$ then letting $\delta$ go to zero gives
	\[ \inf_{\varepsilon} \limsup_{n \to \infty} \frac{1}{n} \log \EE_{\sigma \sim \mu_n} \abs*{ \{ \mb{x} \in \A^n \st (\mb{x}, \mb{y}_n) \in \Omega_k^*(\sigma, \alpha\tup\beta, \varepsilon) \} }  \geq F(T, \alpha^{k_1} \mid \beta^{k_2}) \]
for $k \leq \min(k_1, k_2)$. First take $k_2 \to \infty$, then $k_1 \to \infty$, then take the infimum over $k$. We get
\begin{align*}
	f_\mu(T, \alpha \mid \beta) 
	&\leq \inf_{\varepsilon,k} \limsup_{n \to \infty} \frac{1}{n} \log \EE_{\sigma \sim \mu_n} \abs*{ \{ \mb{x} \in \A^n \st (\mb{x}, \mb{y}_n) \in \Omega_k^*(\sigma, \alpha\tup\beta, \varepsilon) \} } \\
	&= \inf_{\calO \ni (\alpha\beta)^G_*\mu} \limsupinf_{n \to \infty} \frac{1}{n} \log \EE_{\sigma \sim \mu_n} \abs{\{ \mb{x} \in \A^n \st (\mb{x}, \mb{y}_n) \in \Omega(\sigma, \calO) \}}
\end{align*}
where the last line follows because the collection of pseudometrics $\{ d_k^* \st k \in \NN\}$ generates the weak$^*$ topology on $\Prob((\A \times \B)^G)$.

\section{Proof of Theorem \ref{thm:main3}}
\label{sec:main3pf}
By analogy with sofic entropy, we denote $\Sigma \coloneqq \{\mu_n\}_{n=1}^\infty$ and denote the left-hand side of the formula in the theorem statement as $\h_{\Sigma}(\mu, \alpha)$.

Endow $\Prob(\A^G)$ with the metric
	\[ d(\lambda, \nu) \coloneqq \sum_{r = 1}^\infty 2^{-r} d^{\ball{e}{r}}(\lambda, \nu) . \]
Note that this induces the weak* topology (where $\A$ is given the discrete topology and $\A^G$ the product topology).

Writing $\mu_\A = \alpha^G_*\mu \in \Prob(\A^G)$, we then have
	\[ \h_{\Sigma}(\mu, \alpha) = \inf_{\varepsilon>0} \limsup_{n \to \infty} \frac{1}{n} \log \EE_{\sigma \sim \mu_n} \abs{\{ \mb{x} \in \A^n \st d(P_{\mb{x}}^\sigma, \mu_\A) < \varepsilon \}} . \]
We will similarly denote $\mu_\B = \beta^G_* \mu \in \Prob(\B^G)$.

\subsection{Lower bound}
Let $\lambda \in \Prob((\A \times \B)^G)$ be any joining of (the shift systems with respective measures) $\mu_\A$ and $\mu_\B$. Then for any $\mb{x} \in \A^n$ and $\mb{y} \in \B^n$ we have
	\[ d(P_{\mb{x}}^\sigma, \mu_\A) \leq d(P_{(\mb{x},\mb{y})}^\sigma, \lambda), \]
where $d$ is defined on $\Prob((\A \times \B)^G)$ analogously to the definition given on $\Prob(\A^G)$ above. This inequality holds because total variation distance is nonincreasing under pushforwards. Consequently
	\[ \h_{\Sigma}(\mu, \alpha) \geq \inf_{\varepsilon>0} \limsup_{n \to \infty} \frac{1}{n} \log \EE_{\sigma \sim \mu_n} \abs{\{ \mb{x} \in \A^n \st d(P_{(\mb{x},\mb{y}_n)}^\sigma, \lambda) < \varepsilon \}} = f_\lambda(S, \amap \mid \bmap) . \]
Taking the supremum over joinings $\lambda$ gives the lower bound.


\subsection{Upper bound}
For $\varepsilon>0$, let
	\[ \join_\varepsilon \coloneqq \{ \lambda \in \Prob^S((\A \times \B)^G) \st d(\amap^G_* \lambda, \mu_\A) < \varepsilon \text{ and } d(\bmap^G_* \lambda, \mu_\B) < \varepsilon \} \]
be the set of shift-invariant ``approximate joinings'' of $\mu_\A$ and $\mu_\B$. Since $\Prob((\A \times \B)^G)$ is compact, for each $\varepsilon>0$ there exist $\lambda_1, \ldots, \lambda_m \in \join_\varepsilon$ such that
	\[ \join_\varepsilon \subseteq \bigcup_{i=1}^m \ball{\lambda_i}{\varepsilon} . \]
By definition of $\mu_n$ we have $\PP_{\sigma \sim \mu_n} ( d(P_{\mb{y}_n}^\sigma, \mu_\B) < \varepsilon) = 1$ for all large enough $n$. Therefore
\begin{align*}
	\h_{\Sigma}(\mu, \alpha)
	&= \inf_{\varepsilon} \limsup_{n \to \infty} \frac{1}{n} \log \EE_{\sigma \sim \mu_n} \abs{\{ \mb{x} \in \A^n \st P_{(\mb{x},\mb{y}_n)}^\sigma \in \join_\varepsilon \}} \\
	&\leq \inf_{\varepsilon} \limsup_{n \to \infty} \frac{1}{n} \log \sum_{i=1}^m \EE_{\sigma \sim \mu_n} \abs{\{ \mb{x} \in \A^n \st P_{(\mb{x},\mb{y}_n)}^\sigma \in \ball{\lambda_i}{\varepsilon} \}} \\
	&= \inf_{\varepsilon} \max_{1 \leq i \leq m} \limsup_{n \to \infty} \frac{1}{n} \log \EE_{\sigma \sim \mu_n} \abs{\{ \mb{x} \in \A^n \st P_{(\mb{x},\mb{y}_n)}^\sigma \in \ball{\lambda_i}{\varepsilon} \}} \\
	&\leq \inf_{\varepsilon} \sup_{\lambda \in \join_\varepsilon} \limsup_{n \to \infty} \frac{1}{n} \log \EE_{\sigma \sim \mu_n} \abs{\{ \mb{x} \in \A^n \st P_{(\mb{x},\mb{y}_n)}^\sigma \in \ball{\lambda}{\varepsilon} \}} .
\end{align*}
Note that the entire expression in the inf is decreasing as $\varepsilon \to 0$, so we may replace the inf with a limit. Rather than taking a continuous limit we write
	\[ \h_{\Sigma}(\mu, \alpha) \leq \lim_{m \to \infty} \sup_{\lambda \in \join_{1/m}} \limsup_{n \to \infty} \frac{1}{n} \log \EE_{\sigma \sim \mu_n} \abs{\{ \mb{x} \in \A^n \st P_{(\mb{x},\mb{y}_n)}^\sigma \in \ball{\lambda}{1/m} \}} . \]

For each $m$ pick $\lambda_m \in \join_{1/m}$ to get within $1/m$ of the supremum. Then the right-hand side is equal to
	\[ \lim_{m \to \infty} \limsup_{n \to \infty} \frac{1}{n} \log \EE_{\sigma \sim \mu_n} \abs{\{ \mb{x} \in \A^n \st P_{(\mb{x},\mb{y}_n)}^\sigma \in \ball{\lambda_m}{1/m} \}} . \tag{$*$} \]
	
Let $\lambda_{m_j}$ be a subsequence with weak* limit $\lambda_0$. By weak* continuity of pushforwards under projection we have $\lambda_0 \in \join(\mu_\A, \mu_\B)$. Now for any $\delta>0$, for all large enough $j$ we have both $1/m_j < \delta/2$ and $d(\lambda_{m_j}, \lambda_0) < \delta/2$, so by the triangle inequality
	\[ \ball{\lambda_{m_j}}{1/m_j} \subseteq \ball{\lambda_0}{\delta}. \]
It follows that the expression in $(*)$, and hence $h_{\Sigma}(\alpha)$, is bounded above by
	\[ \limsup_{n \to \infty} \frac{1}{n} \log \EE_{\sigma \sim \mu_n} \abs{\{ \mb{x} \in \A^n \st P_{(\mb{x},\mb{y}_n)}^\sigma \in \ball{\lambda_0}{\delta} \}} . \]
Taking the infimum over $\delta$ shows that
	\[ h_{\Sigma}(\mu, \alpha) \leq f_{\lambda_0}(S, \amap \mid \bmap) \leq \sup_{\lambda \in \join(\mu_\A, \mu_\B)} f_\lambda (S, \amap \mid \bmap) . \]

\section{Proof of Proposition \ref{prop:sofic}}
\label{sec:soficpf}

All sequences of interest are of the form
	\[ \mu_n = \SBM(\sigma_n, \mb{y}_n, m_n) = \Unif(\{ \sigma \in \Hom(G, \Sym(n)) \st W_{\sigma, \mb{y}_n^{m_n}} = W_n \}) \]
with $\mb{y}_n \in \B^n$, $\sigma_n \in \Sym(n)$, $m_n = o(\log \log n)$, and where $W_n$ is the $\B^{\ball{e}{m_n}}$-weight $W_{\sigma_n, \mb{y}_n^{m_n}}$. In the case of Theorem \ref{thm:main1} we simply have $m_n = 0$ for all $n$.

The theorem will follow from the following:

\begin{lemma}
	Let $\zeta_n$ denote the uniform measure on $\Hom(G, \Sym(n))$. Then for any finite $D \subset G$ and $\delta > 0$ there exists $\varepsilon>0$ such that
		\[ \PP_{\sigma \sim \zeta_n} ( \sigma \text{ is } (D,\delta) \text{-sofic} ) \geq 1 - n^{-\varepsilon n} \]
	for all large enough $n$. \qed
\end{lemma}

This can be proven by making superficial changes to the proof of the similar result in \cite{bowen2020}.

To prove Proposition \ref{prop:sofic}, it now suffices to show that for any $\varepsilon > 0$
	\[ \PP_{\sigma \sim \zeta_n} ( W_{\sigma, \mb{y}_n^{m_n}} = W_n ) \geq n^{-\varepsilon n} \]
for all large enough $n$. To do this, first note that the left-hand side here depends only on the vector  $p_{\mb{y}_n} \in \Prob(\B)$ of letter frequencies. Therefore
\begin{align*}
	\PP_{\sigma \sim \zeta_n} ( \exists \mb{y} \in \B^n  \text{ s.t. } W_{\sigma, \mb{y}^{m_n}} = W_n )
		&\leq \sum_{\mb{y} \st p_{\mb{y}} = p_{\mb{y}_n}} \PP_{\sigma \sim \zeta_n} ( W_{\sigma, \mb{y}^{m_n}} = W_n ) \\
		&= \exp\{ n \shent(p_{\mb{y}_n}) + o(n)\} \PP_{\sigma \sim \zeta_n} ( W_{\sigma, \mb{y}_n^{m_n}} = W_n ).
\end{align*}
But by Proposition \ref{prop:weightballapprox}, if $\sigma \in \Hom(G, \Sym(n))$ and $\mb{Y} \in (\B^{\ball{e}{m_n}})^n$ are such that $W_{\sigma, \mb{Y}} = W_n = W_{\sigma_n, \mb{y}_n^{m_n}}$, then the projection $\mb{Y}_e \in \B^n$ satisfies $(\mb{Y}_e)^{m_n} = \mb{Y}$. Therefore for each $\sigma$
	\[ \abs*{\{\mb{Y} \in (\B^{\ball{e}{m_n}})^n  \st W_{\sigma, \mb{Y}} = W_n\}} = \abs*{\{\mb{y} \in \B^n  \st W_{\sigma, \mb{y}^{m_n}} = W_n\}} . \]
Hence
\begin{align*}
	\EE_{\sigma \sim \zeta_n} \abs*{\{\mb{Y} \in (\B^{\ball{e}{m_n}})^n  \st W_{\sigma, \mb{Y}} = W_n \}}
		&= \EE_{\sigma \sim \zeta_n} \abs*{\{\mb{y} \in \B^n  \st W_{\sigma, \mb{y}^{m_n}} = W_n \}} \\
		&\leq \abs{\B}^n \PP_{\sigma \sim \zeta_n} ( \exists \mb{y} \in \B^n  \text{ s.t. } W_{\sigma, \mb{y}^{m_n}} = W_n ).
\end{align*}
		 
Combining these last few statements, we see that 
	\[ \PP_{\sigma \sim \zeta_n} ( W_{\sigma, \mb{y}_n^{m_n}} = W_n ) \geq \exp\{ - 2n \log \abs{\B} + o(n)\} \EE_{\sigma \sim \zeta_n} \abs*{\{\mb{Y} \in (\B^{\ball{e}{m_n}})^n  \st W_{\sigma, \mb{Y}} = W_n \}} . \]
We can ignore the first factor here since it only decays exponentially fast. By Proposition \ref{prop:Zbounds},
	\[ \EE_{\sigma \sim \zeta_n} \abs*{\{\mb{Y} \in (\B^{\ball{e}{m_n}})^n  \st W_{\sigma, \mb{Y}} = W_n \}} = \frac{Z_n(W_n)}{(n!)^r} \geq (3 \sqrt{n})^{-r \abs{\B^{\ball{e}{m_n}}}^2} e^{F(W_n) n} n^{(1-r)/2} . \]
The third factor is clearly not a problem and can also be ignored. For the first factor,
	\[ \frac{1}{n \log n} \log (3 \sqrt{n})^{-r \abs{\B^{\ball{e}{m_n}}}^2}
		= -r \frac{\abs{\B^{\ball{e}{m_n}}}^2}{n} \frac{\log 3 \sqrt{n}}{\log n}
		\to 0 \text{ as } n \to \infty \]
using Lemma \ref{lem:expbound}. For the second factor, first note that by definition of $F(W_n)$ we have
\begin{align*}
	F(W_n) &= (1-2r) \shent\big(W_n(\cdot)\big) + \sum_{i \in [r]} \shent\big( W_n(\cdot, \cdot; i) \big) \\
		&\geq -2r \shent \big( W_n(\cdot) \big) \\
		&\geq -2r \log \abs*{\B^{\ball{e}{m_n}}}.
\end{align*}
So
	\[ \frac{1}{n \log n} \log e^{F(W_n) n} = \frac{F(W_n)}{\log n} \geq -2r \frac{\log \abs*{\B^{\ball{e}{m_n}}} }{\log n} \to 0 \text{ as } n \to \infty, \]
again using Lemma \ref{lem:expbound}. This implies that for every $\varepsilon>0$ we have
	\[ (3 \sqrt{n})^{-r \abs{\B^{\ball{e}{m_n}}}^2} e^{F(W_n) n} \geq n^{-\varepsilon n} \]
for all large enough $n$, which implies the result.

\section{Proof of Lemma \ref{lem:denomnapprox}}
\label{sec:denomnapproxproof}

We show how to construct a denominator-$n$ weight $W_{\A\B}$ that has a given $\B$-marginal $W_\B$ and is close to a given $(\A \times \B)$-weight $W$ whose $\B$-marginal $\pi_\B W$ is close to $W_\B$. As in the theorem statement, we assume
	\[ d(\pi_\B W, W_\B) < \delta . \]
To minimize the appearance of factors of $\frac{1}{2}$, in this section we work with the $\ell^1$ distance on weights, which is twice the distance defined above. Therefore the previous assumption becomes
	\[ d_1(\pi_\B W, W_\B) = \sum_{i \in [r]} \sum_{b,b' \in \B} \abs{ \pi_\B W(b,b'; i) - W_\B (b,b'; i)} < 2 \delta . \]

We fix distinguished elements $a_0 \in \A$ and $b_0 \in \B$ which will be referred to throughout this section.

\subsection{The vertex measure}
	We first define the weight's vertex measure by
	\begin{align*}
		W_{\A\B}((a,b)) &= \tfrac{1}{n} \floor*{n \cdot W((a,b))}
			& a \in \A \setminus \{a_0\},\ b \in \B \\
		W_{\A\B}((a_0,b)) &= W_\B(b) - \sum_{a \ne a_0} W_{\A\B}((a,b))
			& b \in \B \mathrlap{.}
	\end{align*}
	See Table \ref{tab:vertmeasure}.
	
	\begin{table}
		\centering
		\small
		\begin{tabu}{c|ccc}
			& $a_0$ & $a_1$ & $\cdots$ \\
			\tabucline{-}
			$b_0$ & $\rightarrow$ & $\floor{\cdot}$ & $\floor{\cdot}$\\
			\tabucline[0.1pt]{-}
			$b_1$ & $\rightarrow$ & $\floor{\cdot}$ & $\floor{\cdot}$\\
			\tabucline[0.1pt]{-}
			$\vdots$ & $\rightarrow$ & $\floor{\cdot}$ & $\floor{\cdot}$
		\end{tabu}
		\caption{Picking entries of the vertex measure $W_{\A\B}(\cdot)$. First choose entries of the form $W_{\A\B}((a,b))$ for $a \ne a_0$ by rounding down $W((a,b))$, then fill in the first column in a way that guarantees the correct $\B$-marginal.}
		\label{tab:vertmeasure}
	\end{table}
	
	Note that $\abs*{W_{\A\B}((a,b)) - W((a,b))} \leq 1/n$ for $a \ne a_0$ and
	\begin{align*}
		\abs*{W_{\A\B}((a_0,b)) - W((a_0,b))} 
		&\leq \abs*{W_\B(b) - \pi_\B W(b)} + \abs{\A}/n.
	\end{align*}
	Therefore the $\ell^1$ distance between the vertex measures is
	\begin{align*}
		\sum_{a,b} \abs*{W_{\A\B}((a,b)) - W((a,b))}
		&\leq \abs{\A} \abs{\B} / n + \sum_{b \in \B} \big(\abs*{W_\B(b) - \pi_\B W(b)} + \abs{\A}/n \big) \\
		&\leq 2\delta + 2\abs{\A}\abs{\B}/n.
	\end{align*}
		
\subsubsection{Nonnegativity}
The terms defined by rounding down $W$ using the floor function are guaranteed to be nonnegative, but the others are not. In the following we show how to repair any negativity.

Let $-R/n$ denote the sum of all negative terms in the vertex measure. Since $W$ contains only nonnegative terms we have
	\[ \1{W_{\A\B}((a,b)) < 0} \cdot  \abs{W_{\A\B}((a,b))} \leq \abs{W_{\A\B}((a,b)) - W((a,b))} \quad \text{for all } a,b. \]
Therefore
	\[ R/n \leq \sum_{b \in \B}\abs{W_{\A\B}((a_0,b)) - W((a_0,b))} \leq 2\delta + \abs{\A}\abs{\B}/n .\]

Suppose there is some $b \in \B$ such that $W_{\A\B}((a_0,b)) < 0$. Since $W_{\A\B}$ has denominator $n$, we must have $W_{\A\B}((a_0,b)) \leq -1/n$. By construction, we have
	\[ \sum_{a \in A} W_{\A\B}((a,b)) = W_\B(b) \geq 0, \]
so there exists some $a^+ \in \A$ with $W_{\A\B}((a^+,b)) \geq 1/n$. Increase $W_{\A\B}((a_0,b))$ by $1/n$ and decrease $W_{\A\B}((a^+,b))$ by $1/n$.

The number of times we must repeat this step before all terms are nonnegative is exactly $R$, and each step moves the measure by $\ell^1$ distance $2/n$; therefore the final edited vertex measure is distance at most $2R/n$ from the original $W_{\A\B}$. If we now let $W_{\A\B}$ denote the new, nonnegative vertex measure, by the above bound on $R/n$ we get
	\[ \sum_{a,b} \abs*{W_{\A\B}((a,b)) - W((a,b))} \leq 6 \delta + 4 \abs{\A}\abs{\B}/n . \]

\subsection{The $\B$ half-marginal}
	For the purposes of this construction we use the $\B$ ``half-marginal,'' which we denote
	\begin{align*}
		W(b, (a', b'); i) &\coloneqq \sum_{a \in \A} W((a, b), (a', b'); i) .
	\end{align*}
	This is an element of $\Prob\big( (\B \times (\A \times \B) )^r \big)$.
	
	Before constructing the edge measure of $W_{\A\B}$, in this section we first construct what will be its half-marginal.
	
	For each $i \in [r]$, $b,b' \in \B$, and $a' \in \A$ we define
	\begin{align}
		\label{eqn:halfmarg1}
		W_{\A\B}(b, (a', b'); i) &= \tfrac{1}{n} \floor*{n \cdot W(b, (a', b'); i)} & \text{for } a' \ne a_0,\ b \ne b_0, \\[0.2cm]
		\label{eqn:halfmarg2}
		W_{\A\B}(b, (a_0, b'); i) &= W_\B(b, b'; i) - \sum_{a' \ne a_0} W_{\A\B}(b, (a', b'); i) & \text{for } b \ne b_0, \\[0.2cm]
		\label{eqn:halfmarg3}
		W_{\A\B}(b_0, (a', b'); i) &= W_{\A\B}((a',b')) - \sum_{b \ne b_0} W_{\A\B}(b, (a', b'); i) .
	\end{align}
	See Table \ref{tab:halfmarg} for a representation of which terms are defined by each equation.
	\begin{table}
		\centering
		\small
		\newlength{\len}
		\setlength{\len}{0.8pt}
%
%
%
		\begin{tabu}{c| ccc | ccc | ccc}
			 &
				 $(a_0,b_0)$ & $(a_1, b_0)$ & $(a_2, b_0)$ &
				 $(a_0,b_1)$ & $(a_1, b_1)$ & $(a_2, b_1)$ &
				 $(a_0,b_2)$ & $(a_1, b_2)$ & $(a_2, b_2)$ \\
			 
			\tabucline{-}
			$b_0$ &
				$\downarrow$ & $\downarrow$ & $\downarrow$ &
				$\downarrow$ & $\downarrow$ & $\downarrow$ &
				$\downarrow$ & $\downarrow$ & $\downarrow$ \\
				
			\tabucline{-}
			$b_1$ &
				$\rightarrow$ & $\floor{\cdot}$ & $\floor{\cdot}$ &
				$\rightarrow$ & $\floor{\cdot}$ & $\floor{\cdot}$ &
				$\rightarrow$ & $\floor{\cdot}$ & $\floor{\cdot}$ \\
				
			\tabucline{-}
			$b_2$ &
				$\rightarrow$ & $\floor{\cdot}$ & $\floor{\cdot}$ &
				$\rightarrow$ & $\floor{\cdot}$ & $\floor{\cdot}$ &
				$\rightarrow$ & $\floor{\cdot}$ & $\floor{\cdot}$
		\end{tabu}
		\caption[Table]{A diagram of how the half-marginal $W_{\A\B}(\cdot, (\cdot, \cdot); i)$ is chosen if $\A = \{a_0, a_1, a_2\}$ and $\B = \{b_0, b_1, b_2\}$. First obtain the entries marked $\floor{\cdot}$ by rounding down $W$. Then choose the entries marked $\rightarrow$ according to Equation \ref{eqn:halfmarg2} which ensures that the $\B$-marginal is $W_\B$. Then choose the entries marked $\downarrow$ according to Equation \ref{eqn:halfmarg3} which ensures that the vertex weight is the one we chose above. }
		\label{tab:halfmarg}
	\end{table}
	
	The definition of the terms in (\ref{eqn:halfmarg3}) ensures that
		\[ \sum_{b \in \B} W_{\A\B}(b, (a', b'); i) = W_{\A\B}((a',b')) \quad \text{for all } a',b',i. \]
	This will ensure that $W_{\A\B}$ has the correct vertex measure.
	Note also that by line (\ref{eqn:halfmarg2})
		\[ \sum_{a' \in \A} W_{\A\B}(b, (a', b'); i) = W_\B(b, b'; i) \quad \text{for all } b \in \B \text{ and } b' \in \B \setminus \{b_0\} . \]
	Using this and definition (\ref{eqn:halfmarg3}) we also get
	\begin{align*}
		\sum_{a' \in \A} W_{\A\B}(b_0, (a', b'); i) 
		&= W_\B(b_0, b'; i).
	\end{align*}
	This will ensure that the $\B$-marginal of $W_{\A\B}$ is $W_\B$.
	
	We show now that the half-marginal $W_{\A\B}(\cdot, (\cdot, \cdot); i)$ is $\ell^1$-close to $W(\cdot, (\cdot, \cdot); i)$ by considering separately the contributions to the $\ell^1$ distance from terms defined using Equations \ref{eqn:halfmarg1}, \ref{eqn:halfmarg2}, and \ref{eqn:halfmarg3}.
	\begin{enumerate}
		\item[(\ref{eqn:halfmarg1}) terms:] Each of the terms of $W_{\A\B}$ defined using the floor in equation (\ref{eqn:halfmarg1}) is distance at most $1/n$ from the corresponding term of $W$; therefore the total contribution of these terms to the $\ell^1$ distance is 
		\begin{align*}
			\sum_{\substack{b \in \B \setminus \{b_0\} \\ a' \in \A \setminus\{a_0\}, b' \in \B \\ i \in [r]}} \abs*{W_{\A\B}(b, (a', b'); i) - W(b, (a', b'); i)} 
			&\leq \abs{\A} \abs{\B}^2 r/ n .
		\end{align*}
		
		\item[(\ref{eqn:halfmarg2}) terms:] By the triangle inequality,
		\begin{align*}
			\lhs \abs*{W_{\A\B}(b, (a_0, b'); i) - W(b, (a_0, b'); i)} \\
			&= \abs*{\left(W_\B(b, b'; i) - \sum_{a' \ne a_0} W_{\A\B}(b, (a', b'); i)\right) - \left( \pi_\B W(b, b'; i) - \sum_{a' \ne a_0} W(b, (a', b'); i) \right) } \\
			&\leq \abs*{W_\B(b, b'; i) - \pi_\B W(b, b'; i)} + \sum_{a' \ne a_0} \abs*{W_{\A\B}(b, (a', b'); i) - W(b, (a', b'); i)} .
		\end{align*}
		The total contribution of such terms is therefore
		\begin{align*}
			\lhs \sum_{\substack{b \in \B \setminus \{b_0\},\ b' \in \B \\ i \in [r]}} \abs*{W_{\A\B}(b, (a_0, b'); i) - W(b, (a_0, b'); i)} \\
			&\leq \overbrace{\sum_{\substack{b \in \B \setminus \{b_0\},\ b' \in \B \\ i \in [r]}} \abs*{W_\B(b, b'; i) - (\pi_\B)_*W(b, b'; i)}}^{\leq d_1(W_\B, \pi_\B W)} \\
			&\qquad + \overbrace{\sum_{\substack{b \in \B \setminus \{b_0\}\\ a' \in \A \setminus \{a_0\},\ b' \in \B \\ i \in [r]}} \abs*{W_{\A\B}(b, (a', b'); i) - W(b, (a', b'); i)}}^{=\text{contribution from (\ref{eqn:halfmarg1}) terms}} \\
			&\leq 2\delta + \abs{\A} \abs{\B}^2 r /n .
		\end{align*}
		
		\item[(\ref{eqn:halfmarg3}) terms:]
		Again applying the triangle inequality,
		\begin{align*}
			&\abs*{W_{\A\B}(b_0, (a, b'); i) - W(b_0, (a, b'); i)} \\
			&\quad \leq \abs*{W_{\A\B}((a,b')) - W((a,b'))} + \sum_{b \ne b_0} \abs*{W_{\A\B}(b, (a, b'); i) - W(b, (a, b'); i)}  .
		\end{align*}
		Summing over all $a \in \A$, $b' \in \B$ and $i \in [r]$, we see that the total contribution of such terms is bounded by
		\begin{align*}
			\sum_{\substack{a \in \A, b' \in \B \\ i \in [r]}} &\left[\abs*{W_{\A\B}((a,b')) - W((a,b'))} + \sum_{b \ne b_0} \abs*{W_{\A\B}(b, (a, b'); i) - W(b, (a, b'); i)} \right] \\
			&= \sum_{i \in [r]} \overbrace{\sum_{\substack{a \in \A \\ b \in \B}} \abs*{W_{\A\B}((a,b)) - W((a,b))}}^{\text{vertex measure}} + \overbrace{\sum_{\substack{b \in \B \setminus \{b_0\} \\ a' \in \A \setminus\{a_0\},\ b' \in \B \\ i \in [r]}} \abs*{W_{\A\B}(b, (a', b'); i) - W(b, (a', b'); i)}}^{\text{(\ref{eqn:halfmarg1}) terms}} \\
			&\qquad + \overbrace{\sum_{\substack{b \in \B \setminus \{b_0\}, \ b' \in \B \\ i \in [r]}} \abs*{W_{\A\B}(b, (a_0, b'); i) - W(b, (a_0, b'); i)}}^{\text{(\ref{eqn:halfmarg2}) terms}} \\
			&\leq r \cdot \left[ 6 \delta + 4 \abs{\A} \abs{\B} / n \right] + \left[ \abs{\A} \abs{\B}^2 r / n \right] + \left[ 2\delta + \abs{\A} \abs{\B}^2 r / n \right] \\
			&\leq 8 r \delta + 6 \abs{\A} \abs{\B}^2 r / n .
		\end{align*}
	\end{enumerate}
	Adding up the contributions of the three types of terms, we see that the $\ell^1$ distance between the half-marginals of $W$ and $W_{\A\B}$ is bounded by
		\[ 10r \delta + 8 \abs{\A} \abs{\B}^2 r/n . \]
		
	\subsubsection{Nonnegativity}
	Again, the preceding construction does not guarantee that all terms are nonnegative. In the following we describe how to correct negativity.
	
	Let $-R/n$ be the sum of all negative terms of the half-marginal. As above, we get
		\[ R/n \leq 10r\delta + 7 \abs{\A} \abs{\B}^2r/n . \]
		
	Suppose there is some $b_- \in B$, $(a'_-,b'_-) \in \A \times \B$, and $i \in [r]$ such that $W_{\A\B}(b_-, (a'_-,b'_-); i) < 0$. Then $W_{\A\B}(b_-, (a'_-,b'_-); i) \leq -1/n$. Since 
		\[ \sum_{a' \in \A} W_{\A\B}(b_-, (a',b'_-); i) = W_\B(b_-, b'_-; i) \geq 0 \]
	and
		\[ \sum_{b \in \B} W_{\A\B}(b, (a'_-,b'_-); i) = W_{\A\B}((a'_-,b'_-)) \geq 0 \]
	there exist $a'_+ \in \A$ and $b_+ \in \B$ such that
		\[ W_{\A\B}(b_-, (a'_+,b'_-); i) \geq 1/n \quad \text{and} \quad W_{\A\B}(b_+, (a'_-,b'_-); i) \geq 1/n . \]
	Decrease both of these terms by $1/n$, and increase both $W_{\A\B}(b_-, (a'_-,b'_-); i)$ and $W_{\A\B}(b_+, (a'_+,b'_-); i)$ by $1/n$. This moves the half-marginal by $\ell^1$ distance $4/n$. 
		\[ \sum_{a' \in \A} W_{\A\B}(b, (a',b'); i) = W_\B(b, b'; i) \quad \text{and} \quad \sum_{b \in \B} W_{\A\B}(b, (a',b'); i) = W_{\A\B}((a',b')). \]
		
	This step must be done at most $R$ times to eliminate all negative entries, so the final half-marginal satisfies
	\begin{align*}
		\sum_{i \in [r]} \sum_{b \in \B} \sum_{(a',b') \in \A \times \B} \abs{W_{\A\B}(b, (a',b'); i) - W(b, (a',b'); i)} &\leq (10r\delta + 8 \abs{\A} \abs{\B}^2r/n) + R\cdot 4/n \\
		&\leq 50r\delta + 36 \abs{\A} \abs{\B}^2r/n .
	\end{align*}
	
\subsection{The edge measure}
Finally, we define the edge measure of $W_{\A\B}$ by
	\begin{align}
		\label{eqn:fullmarg1}
			\begin{split}
				W_{\A\B}( (a, b), (a', b'); i) &= \tfrac{1}{n} \floor*{n \cdot W((a, b), (a', b'); i)} \\
				&\hspace{3cm} \text{for } a \ne a_0 \text{ and } (a',b') \ne (a_0, b_0),
			\end{split} \\[0.2cm]
		\label{eqn:fullmarg2}
			\begin{split}
				W_{\A\B}( (a_0, b), (a', b'); i) &= W_{\A\B}(b, (a', b'); i) - \sum_{a \ne a_0} W_{\A\B}((a, b), (a', b'); i) \\
				&\hspace{3cm} \text{for } (a',b') \ne (a_0, b_0),
			\end{split} \\[0.2cm]
		\label{eqn:fullmarg3}
			W_{\A\B}((a, b), (a_0, b_0); i) &= W_{\A\B}((a, b)) - \sum_{(a',b') \ne (a_0,b_0)} W_{\A\B}((a, b), (a', b'); i) .
	\end{align}
	See Table \ref{tab:fullmarg}.
	
	\begin{table}
		\centering
		\small
		\setlength{\len}{0.8pt}
		\begin{tabu}{c | ccc | ccc | ccc }
		
			 &
			 $(a_0,b_0)$ & $(a_1, b_0)$ & $(a_2, b_0)$ &
			 $(a_0,b_1)$ & $(a_1, b_1)$ & $(a_2, b_1)$ &
			 $(a_0,b_2)$ & $(a_1, b_2)$ & $(a_2, b_2)$ \\
			 
			\tabucline{-} 
			$(a_0,b_0)$ & 
				$\rightarrow$ & $\downarrow$ & $\downarrow$ &
				$\downarrow$ & $\downarrow$ & $\downarrow$ &
				$\downarrow$ & $\downarrow$ & $\downarrow$ \\ 
			$(a_1,b_0)$ & 
				$\rightarrow$ & $\floor{\cdot}$ & $\floor{\cdot}$ &
				$\floor{\cdot}$ & $\floor{\cdot}$ & $\floor{\cdot}$ &
				$\floor{\cdot}$ & $\floor{\cdot}$ & $\floor{\cdot}$ \\ 
			$(a_2,b_0)$ & 
				$\rightarrow$ & $\floor{\cdot}$ & $\floor{\cdot}$ &
				$\floor{\cdot}$ & $\floor{\cdot}$ & $\floor{\cdot}$ &
				$\floor{\cdot}$ & $\floor{\cdot}$ & $\floor{\cdot}$ \\ 
				
			\tabucline{-} 
			$(a_0,b_1)$ & 
				$\rightarrow$ & $\downarrow$ & $\downarrow$ &
				$\downarrow$ & $\downarrow$ & $\downarrow$ &
				$\downarrow$ & $\downarrow$ & $\downarrow$ \\ 
			$(a_1,b_1)$ & 
				$\rightarrow$ & $\floor{\cdot}$ & $\floor{\cdot}$ &
				$\floor{\cdot}$ & $\floor{\cdot}$ & $\floor{\cdot}$ &
				$\floor{\cdot}$ & $\floor{\cdot}$ & $\floor{\cdot}$ \\ 
			$(a_2,b_1)$ & 
				$\rightarrow$ & $\floor{\cdot}$ & $\floor{\cdot}$ &
				$\floor{\cdot}$ & $\floor{\cdot}$ & $\floor{\cdot}$ &
				$\floor{\cdot}$ & $\floor{\cdot}$ & $\floor{\cdot}$ \\ 
				
			\tabucline{-} 
			$(a_0,b_2)$ & 
				$\rightarrow$ & $\downarrow$ & $\downarrow$ &
				$\downarrow$ & $\downarrow$ & $\downarrow$ &
				$\downarrow$ & $\downarrow$ & $\downarrow$ \\ 
			$(a_1,b_2)$ & 
				$\rightarrow$ & $\floor{\cdot}$ & $\floor{\cdot}$ &
				$\floor{\cdot}$ & $\floor{\cdot}$ & $\floor{\cdot}$ &
				$\floor{\cdot}$ & $\floor{\cdot}$ & $\floor{\cdot}$ \\ 
			$(a_2,b_2)$ & 
				$\rightarrow$ & $\floor{\cdot}$ & $\floor{\cdot}$ &
				$\floor{\cdot}$ & $\floor{\cdot}$ & $\floor{\cdot}$ &
				$\floor{\cdot}$ & $\floor{\cdot}$ & $\floor{\cdot}$
		\end{tabu}
		\caption{A diagram of how the edge measure $W_{\A\B}((\cdot,\cdot), (\cdot, \cdot); i)$ is chosen if $\A = \{a_0, a_1, a_2\}$ and $\B = \{b_0, b_1, b_2\}$. First obtain the entries marked $\floor{\cdot}$ by rounding down entries of $W$. Then choose entries marked $\downarrow$ according to Equation \ref{eqn:fullmarg2}, which ensures that the $\B$ half-marginal is the one chosen above. Then choose entries marked $\rightarrow$ according to Equation \ref{eqn:fullmarg3}, which ensures that the vertex measure is the one chosen above.}
		\label{tab:fullmarg}
	\end{table}
	
	It follows from this definition that $W_{\A\B}$ is a (signed) weight with $\B$-marginal $W_\B$.

	We now check that $W_{\A\B}$ is $\ell^1$-close to $W$. We consider separately the contribution to the $\ell^1$ distance of terms defined in equations (\ref{eqn:fullmarg1}), (\ref{eqn:fullmarg2}), and (\ref{eqn:fullmarg3}):
	\begin{enumerate}
		\item[(\ref{eqn:fullmarg1}) terms:] Each term of $W_{\A\B}$ defined using the floor function in equation (\ref{eqn:fullmarg1}) is distance at most $1/n$ from the corresponding $W$ term. The total contribution of these terms to the $\ell^1$ distance is therefore at most $\abs{\A}^2 \abs{\B}^2 r/n$.

		\item[(\ref{eqn:fullmarg2}) terms:] Applying the triangle inequality to terms defined in equation (\ref{eqn:fullmarg2}),
		\begin{align*}
			\lhs \abs*{W_{\A\B}( (a_0, b), (a', b'); i) - W( (a_0, b), (a', b'); i)} \\
			&\leq \abs*{W_{\A\B}(b, (a', b'); i) -W(b, (a', b'); i)} \\
			&\qquad + \sum_{a \ne a_0} \abs*{W_{\A\B}((a, b), (a', b'); i) - W((a, b), (a', b'); i)} \\
			&\leq \abs*{W_{\A\B}(b, (a', b'); i) -W(b, (a', b'); i)} + \abs{\A}/n .
		\end{align*}
		By the $\ell^1$ bound on the distance between the half-marginals, the total contribution of all such terms is therefore
		\begin{align*}
			\lhs \sum_{i \in [r]} \sum_{b} \sum_{(a',b') \ne (a_0, b_0)} \left( \abs*{W_{\A\B}(b, (a', b'); i) -W(b, (a', b'); i)} + \abs{\A}/n \right) \\
			&\leq [50r\delta + 36 \abs{\A}^2\abs{\B}^2 r / n] + \abs{\A}^2 \abs{\B}^2 r / n \\
			&= 50r\delta + 37 \abs{\A}^2\abs{\B}^2 r / n
		\end{align*}
			
		\item[(\ref{eqn:fullmarg3}) terms:] Applying the triangle inequality to terms defined in equation (\ref{eqn:fullmarg3}):
		\begin{align*}
			\lhs \abs*{W_{\A\B}((a, b), (a_0, b_0); i) - W_{\A\B}((a, b), (a_0, b_0); i)} \\
			&\leq \abs*{W_{\A\B}((a, b)) - W((a, b))} + \sum_{(a',b') \ne (a_0,b_0)} \abs*{W_{\A\B}((a, b), (a', b'); i) - W((a, b), (a', b'); i)}.
		\end{align*}
		Therefore the total contribution of all such terms is
		\begin{align*}
			\lhs \sum_{i \in [r]} \sum_{a,b} \abs*{W_{\A\B}((a, b), (a_0, b_0); i) - W_{\A\B}((a, b), (a_0, b_0); i)} \\
			&= \sum_{i \in [r]} \sum_{a,b} \Bigg[ \abs*{W_{\A\B}((a,b)) - W((a,b))} \\
			&\qquad + \sum_{(a',b') \ne (a_0,b_0)} \abs*{W_{\A\B}((a, b), (a', b'); i) - W((a, b), (a', b'); i)} \Bigg] \\
			&= \overbrace{\sum_{i \in [r]} \sum_{a,b} \abs*{W_{\A\B}((a,b)) - W((a,b))}}^{\text{vertex measure}} \\
			&\qquad + \overbrace{\sum_{i \in [r]} \sum_{a \ne a_0}\sum_{b}\sum_{(a',b') \ne (a_0,b_0)} \abs*{W_{\A\B}((a, b), (a', b'); i) - W((a, b), (a', b'); i)}}^{\text{\eqref{eqn:fullmarg1} terms}} \\
			&\qquad + \overbrace{\sum_{i \in [r]} \sum_{b}\sum_{(a',b') \ne (a_0,b_0)} \abs*{W_{\A\B}((a_0, b), (a', b'); i) - W((a_0, b), (a', b'); i)}}^{\text{\eqref{eqn:fullmarg2} terms}} \Bigg] \\
			&\leq r \cdot \left[6 \delta + 3 \abs{\A} \abs{\B}/n \right] + \left[ \abs{\A}^2\abs{\B}^2 r / n \right] + \left[ 50 r\delta + 37 \abs{\A}^2\abs{\B}^2 r / n \right] \\
			&\leq 56 r\delta + 41 \abs{\A}^2\abs{\B}^2r/n.
		\end{align*}
	\end{enumerate}
	
	Summing up the contributions from terms of all three types, we get that
		\[ d_1(W_{\A\B}, W) \leq 106 r\delta + 79 \abs{\A}^2\abs{\B}^2r/n . \]
	
\subsubsection{Nonnegativity}
	We can modify a solution with negative entries to get a nonnegative one similarly to above. Let $-R/n$ be the sum of all negative entries; then
		\[ R/n \leq 106 r\delta + 78 \abs{\A}^2\abs{\B}^2r/n . \]
	
	Suppose there is some entry
		\[ W_{\A\B}((a_-, b_-), (a'_-, b'_-); i) \leq -1/n . \]
	We want to increment this term by $1/n$ without affecting the vertex measure or the $\B$ marginal.
	Since
		\[ \sum_{(a',b') \in \A \times \B} W_{\A\B}((a_-, b_-), (a', b'); i) = W_{\A\B}((a_-, b_-)) \geq 0 \]
	there exists some $(a'_+, b'_+) \in \A \times \B$ such that $W_{\A\B}((a_-, b_-), (a'_+, b'_+); i) \geq 1/n$; similarly since
		\[ \sum_{a \in A} W_{\A\B}((a, b_-), (a', b'_-); i) = W_{\A\B}(b_-, (a'_-,b'_-); i) \geq 0 \]
	there exists some $a_+$ such that $W_{\A\B}((a_+, b_-), (a'_-, b'_-); i) \geq 1/n$.
	Increase 
		\[ W_{\A\B}((a_-, b_-), (a'_-, b'_-); i) \quad \text{and} \quad W_{\A\B}((a_+, b_-), (a'_+, b'_+); i) \]
	by $1/n$, and decrease
		\[ W_{\A\B}((a_-, b_-), (a'_+, b'_+); i) \quad \text{and} \quad W_{\A\B}((a_+, b_-), (a'_-, b'_-); i) \]
	by $1/n$. This moves the weight by $\ell^1$ distance $4/n$.
	
	Since $R$ is the maximum number of times we need to do this before there are no more negative entries, the final weight satisfies
	 	\[ d_1(W_{\A\B}, W) \leq 106 r\delta + 79 \abs{\A}^2\abs{\B}^2r/n + 4R/n \leq 530 r\delta + 391 \abs{\A}^2\abs{\B}^2r/n . \]
	To simplify, we write
		\[ d_1(W_{\A\B}, W) \leq 530r ( \delta + \abs{\A \times \B}^2/n) , \]
	or
		\[  d(W_{\A\B}, W) \leq 265 r( \delta + \abs{\A \times \B}^2/n). \]

\bibliographystyle{alpha}
\bibliography{references}

\end{document}